\documentclass[a4paper,12pt]{amsart}

\usepackage{amssymb,xspace,amscd,yhmath,mathdots,txfonts,
mathrsfs, youngtab,stmaryrd}

\usepackage[matrix,arrow,curve]{xy}\CompileMatrices
\usepackage{hyperref}
\usepackage{yfonts}[1998/10/03]

\DeclareMathAlphabet{\mathpzc}{OT1}{pzc}{m}{it}

\numberwithin{equation}{section}

\theoremstyle{plain}

\newtheorem{thm}{Theorem}[section]

\newtheorem{lem}[thm]{Lemma}

\newtheorem{cor}[thm]{Corollary}

\newtheorem{prop}[thm]{Proposition}
\newtheorem{ntz}{Notation}[section]

\theoremstyle{definition}
\newtheorem{dfn}{Definition}[section]

\newtheorem{defn}{Definition}[section]
\newtheorem{exam}[thm]{Example}
\newtheorem{nnott}{Notation}[section]

\newtheorem{rmk}[thm]{Remark}

\DeclareMathAlphabet{\mathpzc}{OT1}{pzc}{m}{it}

\newcommand\SO{\mathbf{SO}}
\newcommand\ot{\mathfrak{o}}

\DeclareMathOperator{\R}{{\mathbb{R}}}

\DeclareMathOperator{\kt}{{\kappaup}}

\DeclareMathOperator{\Gf}{\mathbf{G}}

\DeclareMathOperator{\supp}{\mathrm{supp}}
\DeclareMathOperator{\ad}{\mathrm{ad}}

\DeclareMathOperator{\st}{\mathfrak{s}}

\DeclareMathOperator{\Z}{\mathbb{Z}}

\DeclareMathOperator{\at}{\mathfrak{a}}
\DeclareMathOperator{\pt}{\mathfrak{p}}
\DeclareMathOperator{\ft}{\mathfrak{f}}

\DeclareMathOperator{\Ht}{\textfrak{H}}

\DeclareMathOperator{\C}{\mathbb{C}}

\newcommand\Qf{\mathbf{Q}}

\newcommand\qt{\mathfrak{q}}
\newcommand\gt{\mathfrak{g}}

\newcommand\Rad{\mathpzc{R}}
\newcommand\Qq{\texttt{Q}}

\newcommand\hg{\mathfrak{h}}

\newcommand\Hq{\mathpzc{H}}

\newcommand\Tq{\mathpzc{T}}

\newcommand\Jd{\mathrm{J}}
\newcommand\qq{\mathpzc{q}}
\newcommand\pq{\mathpzc{p}}

\newcommand\SL{\mathbf{SL}}
\newcommand\SU{\mathbf{SU}}
\newcommand\CP{\mathbb{CP}}

\newcommand\Sb{\mathbf{S}}

\newcommand\Id{\mathrm{I}}

\newcommand\Bf{\mathbf{B}}
\newcommand\bt{\mathfrak{b}}

\newcommand\gl{\mathfrak{gl}}

\newcommand\slt{\mathfrak{sl}}
\newcommand\su{\mathfrak{su}}

\newcommand\sfE{\textsf{E}}

\newcommand{\Bz}{\mathpzc{B}}

\newcommand\id{\mathrm{id}}

\newcommand\vq{\mathpzc{v}}
\newcommand\wq{\mathpzc{w}}
\newcommand\uq{\mathpzc{u}}

\newcommand\bil{\textswab{b}}

\newcommand\Hb{\mathbb{H}}
\newcommand\Qe{\mathpzc{Q}}

\newcommand\aq{\mathpzc{a}}

\newcommand\sfV{\textsf{V}}
\newcommand\sfW{\textsf{W}}
\newcommand\Wi{\mathpzc{W}}

\newcommand\epi{\epsilon}

\newcommand\K{\mathbb{K}}

\newcommand\Ta{\textsf{T}}
\newcommand\Ha{\textsf{H}}
\newcommand\sfM{\textsf{M}}
\newcommand\sfF{\textsf{F}}

\newcommand\sfL{\textsf{L}}

   \def\DHLhksqrt#1#2{\setbox0=\hbox{$#1\sqrt{#2\,}$}\dimen0=\ht0
     \advance\dimen0-0.2\ht0
     \setbox2=\hbox{\vrule height\ht0 depth -\dimen0}%
     {\box0\lower0.4pt\box2}}

\title{Higher order Levi forms on  homogeneous CR manifolds}
\author{S.Marini, C.Medori, M.Nacinovich}
\address{Stefano Marini: Dipartimento di Scienze Matematiche, Fisiche e Informatiche\\ Universit\`a di Parma\\ Parco Area delle Scienze 53/A (Campus), 43124 Parma
 (Italy)} \email{stefano.marini@unipr.it}

\address{Costantino Medori:
Dipartimento di Scienze Matematiche, Fisiche e Informatiche\\ Universit\`a di Parma\\ Parco Area delle Scienze 53/A (Campus), 43124 Parma
 (Italy)} \email{costantino.medori@unipr.it}

\address{Mauro Nacinovich:
Dipartimento di Matematica\\ II Universit\`a di Roma
``Tor Ver\-ga\-ta''\\ Via della Ricerca Scientifica\\ 00133 Roma
(Italy)}
\email{nacinovi@mat.uniroma2.it}


\date\today

\begin{document}
\begin{abstract}
 We investigate the  nondegeneracy 
 of higher order Levi forms on
 weakly nondegenerate
 homogeneous $CR$ manifolds. Improving previous results,
 we prove that general orbits of real forms in complex flag
 manifolds have
order less or equal $3$ and the compact ones 
less or equal~$2$. Finally we construct by Lee extensions
weakly nondegenerate 
$CR$ vector bundles with arbitrary orders of nondegeneracy.
\end{abstract}

\date\today
\thanks{\\
This research has financially been supported by \emph{PRIN 2017} ``\textit{Real
and Complex Manifolds: Geometry, Topology and  
Harmonic Analysis}'' and  the Programme \emph{FIL-Quota Incentivante} of University of Parma and co-sponsored by Fondazione Cariparma.}

\maketitle 
\tableofcontents

\section*{Introduction}
The Levi form is 
a basic invariant of $CR$ geometry (see e.g. \cite{BHLN}).
It is a hermitian symmetric form on the space of tangent holomorphic vector
fields, which,
when the $CR$ codimension is larger than one, is vector valued. 
Its nondegeneracy was shown in \cite{Tan70} to be a 
sufficient condition
to apply Cartan's method to investigate equivalence and
automorphisms of $CR$ structures and
is an obvious obstruction for locally representing the manifold
as a product of a $CR$ manifold of smaller dimension and of
a nontrivial complex manifold.  
Sufficient more general conditions  
preventing a $CR$ manifold $\sfM$ from being foliated
by complex leaves of positive dimension 
or from having an infinite dimensional group of local $CR$ automorphisms 
can be expressed by
the nondegeneracy  of higher order Levi forms (see e.g. \cite{Freeman1977}).
In the case of 
homogeneous $CR$ manifolds these properties 
can be rephrased in terms of their
associated $CR$ algebras and lead to the notions of
\textit{weak nondegeneracy} and \textit{ideal nondegeneracy}
in \cite{MN05}. The last one was renamed \textit{contact nondegeneracy}
and proved sufficient for the finite dimensionality of the group of
$CR$ automorphisms in \cite{NMSM}.\par
Iterations of the Levi forms can be described by building  
descending chains of algebras of vector fields, whose lengths 
can be taken as a measure of nondegeneracy (see \S\ref{s1.1}).
One of these numbers, that we call here \emph{Levi order},
and relates to weak nondegeneracy, is the main topic of this paper.
The real submanifolds $\sfM$
of a complex flag manifold $\sfF$ of a semisimple complex group
$\Sb$ that are orbits 
of its real form $\Sb_{\R}$ 
form an interesting class
of homogeneous $CR$ manifolds,  that has  been studied 
e.g. in \cite{AMN06,AMN06b}.
In \cite{fels07} G.~Fels showed that when the isotropy $\Qf$ 
of $\sfF$ is a maximal parabolic subgroup, 
and $\sfM$ is weakly nondegenerate, 
then its Levi order
is at most  $3$ and found an example where it is in fact
equal to $3$. In \S\ref{sec2.2} we prove that this bound is 
valid for general weakly nondegenerate 
real orbits, dropping the maximality assumption
on $\Qf$ 
and give further examples of weakly nondegenerate real orbits
having Levi order $3$. Moreover we show that the minimal
orbit (the single one which is compact, cf. \cite{Wolf69})
cannot have a finite Levi order larger than $2$ and that
the same result is valid for a larger class or orbits, that
we name \textit{of the minimal type}. Orbits which are not
of the minimal type may have any finite order $1,2,3$.
Our methods are illustrated by several examples.
We point out that, together with the new results obtained here,  
those in \cite{AMN06}, where descriptions in terms of
cross-marked Satake diagrams are emphasised, 
would allow to list  
the minimal orbits of Levi orders
$1$ and $2$.\par
In \cite{fels07} G.~Fels posed the question of the existence
of weakly nondegenerate homogeneous $CR$ manifolds with
Levi order larger than $3$. In \S\ref{s3}
we exhibit,    
by constructing
some
$CR$ vector bundles over $\CP^{1}$, 
weakly nondegenerate homogeneous $CR$ manifolds
having Levi order $\qq,$ for every positive integer $\qq$.

\section{Nondegeneracy  conditions} 
\subsection{\textsc{Abstract $CR$ manifolds}}\label{s1.1}
In this subsection we discuss  some
notions of nondegeneracy 
for
general smooth abstract $CR$ manifolds of type $(n,k).$ 
We will eventually be interested in the locally homogeneous case and therefore,
in the rest of this section,  
in their reformulation in the framework of Lie algebras theory.
\par
We recall that an abstract $CR$ manifold of type $(n,k)$ is
defined by the datum, on a
smooth manifold $\sfM$ of real dimension $2n{+}k,$ of a rank $n$ 
smooth complex linear subbundle
$\Ta^{0,1}\sfM$ of its complexified tangent bundle $\Ta^{\C}\sfM,$ satisfying 
\begin{equation}
 \Ta^{0,1}\sfM\cap\overline{\Ta^{0,1}\sfM}=\{0\}
\end{equation}
and the formal
integrability condition 
\begin{equation}\label{eq1.2}
 [\Gamma^{\infty}(\sfM,\Ta^{0,1}\sfM),\Gamma^{\infty}(\sfM,\Ta^{0,1}\sfM)]\subseteq
 \Gamma^{\infty}(\sfM,\Ta^{0,1}\sfM).
\end{equation}
Set 
\begin{equation}
 \Ta^{1,0}\sfM=\overline{\Ta^{0,1}\sfM},\;\; \Ha^{\C}\sfM=\Ta^{1,0}\sfM\oplus\Ta^{0,1}\sfM,\;\;
 \Ha\sfM=\Ha^{\C}\sfM\cap\Ta\sfM.
\end{equation}
The rank $2n$ real subbundle $\Ha\sfM$ of $\Ta\sfM$ is the \emph{real contact distribution}
underlying the $CR$ structure of $\sfM.$ 

A smooth $\R$-linear bundle map $\Jd_{\sfM}:\Ha\sfM\to\Ha\sfM$ is defined by the equation 
\begin{equation}\label{eq1.4}
 \Ta^{0,1}\sfM=\{\vq+i\Jd_{\sfM}{\vq}\mid \vq\in\Ha\sfM\}.
\end{equation}
The map $\Jd_{\sfM}$ squares to $-\Id_{\Ha}$ and is the \emph{partial complex structure}
of $\sfM.$ \par
An equivalent definition of the $CR$ structure can be given by assigning first an even dimensional
distribution $\Ha\sfM$ and then a smooth partial complex structure $\Jd_{\sfM}$ on $\Ha\sfM$ 
in such a way that
the complex distribution \eqref{eq1.4} satisfies~\eqref{eq1.2}. \par 

\par \smallskip
Let us denote by $\Hq$ (resp. $\Tq,$ $\Hq^{\C}$, $\Tq^{0,1},$ $\Tq^{1,0}$)
the sheaf of germs of smooth sections of $\Ha\sfM$ 
(resp. $\Ta\sfM,$ $\Ha^{\C}\sfM,$ $\Ta^{0,1}\sfM,$ $\Ta^{1,0}\sfM$). 
\begin{defn}
 A $CR$ manifold $\sfM$ is called \emph{fundamental} at its point $x$ 
  if $\Hq_{x}$
 generates the Lie algebra $\Tq_{x}.$ 
\end{defn}
We define recursively a nested sequence of sheaves of germs of
smooth complex valued  vector fields on $\sfM$ 
\begin{equation}
 \Tq^{0,1}_{0}\supseteq\Tq^{0,1}_{1}\supseteq\cdots\supseteq\Tq^{0,1}_{\pq}
 \supseteq\Tq^{0,1}_{\pq+1}\supseteq\cdots
\end{equation}
by setting 
\begin{equation} 
\begin{cases}
 \Tq^{0,1}_{0}{=}\Tq^{0,1},\\
  \begin{aligned}
 \Tq^{0,1}_{\pq}{=}\bigsqcup_{{x\in\sfM}}\left\{Z\in{\Tq^{0,1}_{\pq-1}}_{x}\left| \; [Z,\Tq^{1,0}_{x}]
 {\subseteq}\Tq^{0,1}_{\pq-1}{}_{x}+\Tq^{1,0}_{x}\right\}\right., 
 \;\text{for $\pq{\geq}1.$}\end{aligned}
\end{cases}
\end{equation}
By conjugation we obtain another nested sequence of sheaves 
\begin{equation}
 \Tq^{1,0}_{0}\supseteq\Tq^{1,0}_{1}\supseteq\cdots\supseteq\Tq^{1,0}_{\pq}\supseteq
 \Tq^{1,0}_{\pq+1}\supseteq\cdots
\end{equation}
where
\begin{equation} 
\begin{cases}
 \Tq^{1,0}_{0}{=}\Tq^{1,0},\\
 \begin{aligned}
 \Tq^{1,0}_{\pq}{=}\bigsqcup_{{x\in\sfM}}\left\{Z\in{\Tq^{1,0}_{\pq-1}}_{x}\left| \; [Z,\Tq^{0,1}_{x}]
 {\subseteq}{\Tq^{1,0}_{\pq-1}}_{x}+\Tq^{0,1}_{x}\right\}\right.,
 \;\text{for $\pq{\geq}1.$}\end{aligned}
\end{cases}
\end{equation}
These sequences, considered 
by Freeman in \cite[Thm.3.1]{Free77},
correspond to the chain \eqref{e1.1} that we construct 
in the locally homogeneous case.
\par\medskip
In the same paper (cf. \cite[Remarks 4.5] {Free77}) 
also another sequence, introduced before   in \cite{Gr68, He64},
was considered, 
which will correspond  to \eqref{e1.2}, namely
\begin{equation} 
 \Hq\supseteq\Hq^{1}\supseteq\cdots\supseteq\Hq^{\pq}\supseteq\Hq^{\pq+1}\supseteq\cdots
\end{equation}
{with}
 \begin{equation}\begin{cases}
 \Hq^{0}=\Hq^{\C},\\
 \Hq^{\pq}{=}\bigsqcup_{{x\in\sfM}}
 \left\{X\in\Hq_{x}^{\pq-1}\mid [X,\Hq_{x}]
 \subseteq\Hq^{\pq-1}_{x}\right\},\;\;\text{for 
 $\pq>0.$}
 \end{cases}
\end{equation}
\begin{defn}
 The $CR$ manifold $\sfM$ has, at its point $x,$  
\begin{itemize}
 \item Levi 
 order $\qq$ if  
 $\Tq^{0,1}_{\qq-1}{}_{x}\neq\Tq^{0,1}_{\qq}{}_{x}=\{0\}$;
 \item contact 
 order $\qq$ if  $\Hq^{\qq-1}_{x}\neq\Hq^{\qq}_{x}\,{=}\,\{0\}.$
\end{itemize}
We say that $\sfM$ is at its point $x$ 
\begin{itemize}
 \item \emph{Levi} (resp. \!\emph{contact}) \emph{nondegenerate} if it has Levi (resp. \!contact) order~1;
 \item \emph{weakly} (resp. \!\emph{contact}) \emph{nondegenerate} if it 
 has finite Levi (resp. \!contact) order $\pq{\geq}1$;
 \item \emph{holomorphically} (resp. \!\emph{contact}) 
 \emph{degenerate} if it is not weakly (resp. \!contact)
 nondegenerate.
\end{itemize}
\end{defn}
\par
The Levi order 
at $x$ is the smallest $\qq$ for which,
given any nonzero germ $\bar{Z}\,{\in}\Tq^{0,1}_{x},$ we can find a $\pq{\leq}\qq$ and 
 $Z_{1},\hdots,Z_{\pq}\,{\in}\,
\Tq^{1,0}_{x}$ such that 
\begin{equation*}\tag{$*$}
 [Z_{1},[Z_{2},\hdots,[Z_{\pq},\bar{Z}]]]\notin \Hq^{0}_{x}.
\end{equation*}
The contact order can be defined in the same way, but 
with $Z_{1},\hdots,Z_{\pq}$ taken in $\Hq^{0}_{x}.$ 
We have therefore 
\begin{prop}
 Let us keep the notation introduced above. Fix a point $x$ in $\sfM.$ Then:
\begin{itemize}
 \item $\sfM$ has Levi  order $1$ at $x$ if and only if
 it has contact  order $1$ at $x.$ 
 \item If $\sfM$ has finite  Levi  order $\qq{\geq}2$ at $x,$ then 
 it has also finite contact  order $\qq',$ with 
 $2{\leq}\qq'{\leq}\qq$ at $x.$ \qed
\end{itemize}
\end{prop}
\subsection{\textsc{Homogeneous $CR$ manifolds and 
$CR$ algebras}}
Let $\Gf_{\R}$ be a Lie group of $CR$ diffeomorphisms
acting transitively on a $CR$ manifold $\sfM.$
Fix a point $x$ of $\sfM$ and let $\piup\,{:}\,\Gf_{\R}\,{\ni}\,g\,{\to}\,g{\cdot}x\,{\in}\,\sfM$
be the natural projection. The differential at $x$ defines a map $\piup_{*}:\gt_{\R}\,{\to}\,\Ta_{x}\sfM$ 
of the Lie algebra $\gt_{\R}$ of $\Gf_{\R}$ onto the tangent space to $\sfM$ at $x.$ 
By the formal integrability of the partial complex structure of $\sfM,$
the pullback $\qt\,{=}\,(\piup^{\C}_{*})^{-1}(\Ta^{0,1}_{x}\sfM)$ 
of the space of tangent vectors
of type $(0,1)$ at $x$
by the complexification of the differential 
 is a complex Lie subalgebra $\qt$ of the complexification 
$\gt\,{=}\,\C{\otimes}_{\R}\gt_{\R}$ of $\gt_{\R}.$ Vice versa, the assignment  of a complex
Lie subalgebra $\qt$ of $\gt$  
yields a formally integrable 
$\Gf_{\R}$-equivariant partial complex structure on a locally 
homogeneous space
 $\sfM$ of
$\Gf_{\R}$ by the requirement that $\Ta_{x}^{0,1}\sfM\,{=}\,\piup^{\C}_{*}(\qt)$
(see e.g. \cite{AMN06,MN05}).
These considerations led to the following definition.
\begin{dfn}
 A \emph{ $CR$ algebra} is a pair  $(\gt_{\R},\qt)$, 
consisting of a real Lie algebra $\gt_{\R}$ and  
a complex Lie subalgebra $\qt$ of its complexification
$\gt{=}\C{\otimes}_{\R}\gt_{\R},$ 
 such that the quotient $\gt_{\R}/(\gt_{\R}\cap\qt)$ is a finite dimensional real vector space. \par
 We call the intersection $\qt\cap\gt_{\R}$  its
 \emph{isotropy subalgebra} and say that $(\gt_{\R},\qt)$ is \emph{effective}
 when $\qt\cap\gt_{\R}$ does not contain any nontrivial ideal of $\gt_{\R}.$ 
 \end{dfn}
 If $\Gf_{\R}$ is the real form of a complex Lie algebra $\Gf$ and 
 $\qt$ the Lie algebra of its closed subgroup $\Qf,$  then  $\sfM$
 is locally $CR$ diffeomorphic to 
 the orbit 
 of $\Gf_{\R}$ in the complex homogeneous space $\Gf{/}\Qf$ 
 and its $CR$ structure is \textit{induced} by the complex structure of $\Gf{/}\Qf.$
These considerations can be generalized to
\textit{locally homogeneous $CR$ manifolds} 
 (see e.g. \cite{AMN06}).
\par\smallskip
The CR-dimension and 
codimension of $\sfM$ are expressed 
in terms of its associated $CR$ algebra 
$(\gt_{\R},\qt)$ by 
\begin{equation*} 
\begin{cases}
 CR-\dim_{\C}\sfM=\dim_{\C}\qt-\dim_{\C}(\qt{\cap}\bar{\qt}),\\
 CR-\textrm{codim}\sfM=\dim_{\C}\gt - \dim_{\C}(\qt+\bar{\qt}).
\end{cases}
\end{equation*}
\par \begin{defn} We call \emph{fundamental} a $CR$ algebra $(\gt_{\R},\qt)$ such that 
$\qt{+}\bar{\qt}$ generates $\gt$ as a Lie algebra and we say that it is 
\begin{itemize}
 \item \emph{of complex type} if $\qt{+}\bar{\qt}\,{=}\,\gt,$
 \item \emph{of contact type} if $\qt{+}\bar{\qt}\,{\subsetneqq}\,\gt.$
\end{itemize}
\end{defn}
A corresponding $CR$ manifold $\sfM$ is in the first case a complex manifold
by Newlander-Nirenberg theorem (cf. \cite{AF79,NN}), 
while contact type is equivalent to the fact that its
\textit{$CR$ distribution}  is strongly non-integrable.

\subsection{\textsc{Levi-order of weak nondegeneracy}}
 The \textit{Levi form} is a basic invariant of $CR$ geometry. 
 When $\sfM$ is locally homogeneous, it can be computed by using 
 its associated $CR$ algebra $(\gt_{\R},\qt)$ 
 (for definitions and basic properties, cf.  
 e.g. \cite{BHLN}). Nondegeneracy of the Levi form can be
 stated by 
\begin{equation*}
 \forall{Z}\in\qt{\backslash}\bar{\qt}, \;\;\exists\, Z'\in\bar{\qt} \;\;\;\text{such that}\;\;\; [Z,Z']\notin\qt+\bar{\qt}.
\end{equation*}
This is equivalent to 
\begin{equation*}
 \qt^{{(1)}}\coloneqq\{Z\in\qt\mid [Z,\bar{\qt}]\subseteq\qt+\bar{\qt}\}=\qt\cap\bar{\qt}.
\end{equation*}
When this condition is not satisfied, we say that $(\gt_{\R},\qt)$ is \emph{Levi-degenerate}. 
To measure the degeneracy of the Levi form, one can consider its \textit{iterations}:
in the homogeneous case this means,
given a $Z\,{\in}\,\qt{\backslash}(\qt{\cap}\bar{\qt}),$ to seek whether it is possible to find
$L_{1},\hdots,L_{\pq}\,{\in}\,\bar{\qt}$ such that $[L_{1},\hdots,L_{\pq},Z]\,{\notin}\,
\qt{+}\bar{\qt}.$ To this aim, 
it is  convenient to consider 
the descending chain (see e.g.  \cite{fels07,Freeman1977,NMSM,MN05}) 
\begin{equation}\label{e1.1}
 \begin{cases}
\qt^{(0)}\supseteq\qt^{(1)}\supseteq\cdots\supseteq\qt^{(\pq-1)}\supseteq\qt^{(\pq)}
\supseteq\qt ^{{(p+1)}}\supseteq \cdots,\quad\text{with}\\
\qt^{(0)}=\qt,\;\;
\qt^{(\pq)}=\{Z\in\qt^{(\pq-1)}\mid [Z,\bar{\qt}]\subseteq\qt^{(\pq-1)}+\bar{\qt}\}\;\;\;\text{for $\pq{\geq}1.$}
\end{cases}
\end{equation} \par
Note that $\qt{\cap}\bar{\qt}\,{\subseteq}\,\qt^{(\pq)}$ for all integers $\pq{\geq}0.$
Since by assumption $\qt/(\qt{\cap}\bar{\qt})$ is finite dimensional,
there is a smallest nonnegative 
integer $\qq$ such that $\qt^{(\pq)}\,{=}\,\qt^{(\qq)}$ for all $\pq{\geq}\qq.$ 
\begin{defn}
 We call \eqref{e1.1} the \emph{descending Levi chain} of $(\gt_{\R},\qt).$
\par 
Let $\qq$ be a positive integer.
The $CR$ algebra $(\gt_{\R},\qt)$ is said to be 
\begin{itemize}
 \item \emph{weakly nondegenerate} of 
 \emph{Levi order $\qq$} if $\qt^{(\qq-1)}{\supsetneqq}\,\qt^{(\qq)}
 {=}\,\qt{\cap}\bar{\qt}.$
 \item \emph{strictly nondegenerate} if it is weakly nondegenerate of Levi order $1.$
\end{itemize}\par
 If $\qt^{{(\qq)}}{\neq}\qt{\cap}\bar{\qt}$ for all integers $\qq{>}0,$ we say that 
 $(\gt_{\R},\qt)$ is holomorphically degenerate. 
\end{defn}
\begin{prop}
 The terms $\qt^{(p)}$ of  \eqref{e1.1} are Lie subalgebras of~$\qt.$ 
\end{prop} 
\begin{proof}
 By definition, $\qt^{(0)}\,{=}\,\qt$ is a Lie subalgebra of $\qt.$ If $Z_{1},Z_{2}\,{\in}\,\qt^{(1)},$ then 
\begin{align*}
 [[Z_{1},Z_{2}],\bar{\qt}]\subseteq [Z_{1},[Z_{2},\bar{\qt}]]+
 [Z_{2},[Z_{1},\bar{\qt}]]\subseteq [Z_{1}+Z_{2},\qt+\bar{\qt}]\subseteq\qt+\bar{\qt}
\end{align*}
because $[Z_{i},\qt]\,{\subseteq}\,[\qt,\qt]\,{\subseteq}\,\qt,$ 
and $[Z_{i},\bar{\qt}]\,{\subseteq}\,\qt{+}\bar{\qt}$ by the definition of $\qt^{(1)}.$ 
This shows that $\qt^{(1)}$ is a Lie subalgebra of $\qt.$ \par
Next
we argue by recurrence. Let $\pq{\geq}1$ and assume that $\qt^{(\pq)}$
is a Lie subalgebra of $\qt.$ If $Z_{1},Z_{2}\in\qt^{(\pq+1)},$ then $[Z_{1},Z_{2}]\,{\in}\,\qt^{(\pq)}$ 
by the inductive assumption that $\qt^{(\pq)}$ is a Lie subalgebra and 
\begin{align*}
 [[Z_{1},Z_{2}],\bar{\qt}]\subseteq [Z_{1},[Z_{2},\bar{\qt}]]+
 [Z_{2},[Z_{1},\bar{\qt}]]\subseteq [Z_{1}+Z_{2},\qt^{(\pq)}+\bar{\qt}]\subseteq\qt^{(\pq)}+\bar{\qt},
\end{align*}
showing that also $[Z_{1},Z_{2}]\in\qt^{(\pq+1)}.$ 
This completes the proof. 
\end{proof}
Let us introduce the notation 
\begin{equation} 
 \Ht=\qt+\bar{\qt},\;\;\;\Ht_{\,\R}=\Ht\cap\gt_{\R}.
\end{equation}
 \par  
The 
\emph{weak nondegeneracy} 
defined here is equivalent to the notion 
of 
\cite{MN05}, consisting in the requirement 
that, for a complex Lie subalgebra $\ft$ of~$\gt,$ 
\begin{equation} \label{e1.13}
 \qt\subseteq\ft\subseteq \Ht \;\Longrightarrow \ft=\qt.
\end{equation}
Indeed, it easily follows from \cite[Lemma 6.1]{MN05} that 
\begin{equation}
 \ft=\qt+\bar{\qt}^{(\infty)},\;\;\;\text{with}\;\;\; \qt^{(\infty)}={\bigcap}_{\pq{\geq}0}\qt^{(\pq)}
\end{equation}
 is the largest complex Lie subalgebra $\ft$ of $\gt$ with
 $\qt\,{\subseteq}\,\ft\,{\subseteq}\,\Ht.$  
 \subsection{\textsc{Contact nondegeneracy}}  A less restrictive  
 nondegeneracy condition in terms of iterations of the Levi form can be expressed by requiring that,
 given $Z\,{\in}\,\qt{\backslash}(\qt{\cap}\bar{\qt})$ there are $L_{1},\hdots,L_{\pq}\,{\in}\,
 \Ht$ such that $[L_{1},\hdots,L_{\pq},Z]\,{\notin}\,\Ht.$ 
 For a $CR$ algebra of the contact type
 this is equivalent to
 the fact that any ideal $\at$ of $\gt_{\R}$ that is contained in 
 $\Ht_{\,\R}$ is contained in $\qt\,{\cap}\,\gt_{\R}.$ 
Thus this property was called \emph{ideal nondegeneracy} in
 \cite{MN05}.  \par
  When $(\gt_{\R},\qt)$ is the $CR$ algebra at $x$ of a (locally)
 homogeneous $CR$ manifold $\sfM,$ the subspace  
 $\Ht_{\,\R}$ is the pullback to $\gt_{\R}$ 
 of the real contact distribution associated to the $CR$ structure of $\sfM.$ 
 Thus 
  this notion was renamed \emph{contact nondegeneracy} in \cite{NMSM}.
  It was shown in \cite[Lemma 7.2]{MN05} that, for the
 complexification $\at$ of the largest ideal of $\gt_{\R}$ contained in $\Ht_{\,\R},$ the sum
 $\at{+}\qt\,{\cap}\,\bar{\qt}$ is the limit of the descending chain 
\begin{equation} \label{e1.2}
\begin{cases}
 \qt^{[0]}\supseteq \qt^{[1]}\supseteq\cdots\supseteq\qt^{[\pq-1]}{\supseteq}\,\qt^{[\pq]}{\supseteq}\,
 \qt^{[\pq+1]}\supseteq \cdots,\quad\text{with}\\
 \qt^{[0]}=\Ht,\;\;\;
 \qt^{[\pq]}=\{Z\in\Ht \mid [Z,\Ht]\subseteq\qt^{[\pq-1]}\}, \;\;\text{for $\pq{\geq}1.$}
\end{cases}
\end{equation}
Since $\qt\,{\cap}\,\bar{\qt}\,{\subseteq}\,\qt^{[\pq]}$ 
for all integers $\pq{\geq}0,$ the chain \eqref{e1.2}
stabilizes. 
\begin{rmk} Note that, for $\pq{\geq}1,$  
\begin{equation}
 \qt^{[\pq]}=\{Z\in\qt^{[\pq-1]}\mid [Z,\qt^{[0]}]\subseteq\qt^{[\pq-1]}\}.
\end{equation}
This is true in fact for $\pq{=}1$ and for $\pq{>}2$ follows from 
$\qt^{[\pq-1]}{\subseteq}\, \qt^{[\pq-2]}.$ 
 
\end{rmk}
\begin{defn}
 We call \eqref{e1.2} the \emph{descending contact chain}.
\par 
Let $\qq$ be a positive integer.
 The $CR$ algebra $(\gt_{\R},\qt)$ is said to have 
\begin{itemize}
\item \emph{finite contact order $\qq$} if $\qt^{[\qq-1]}{\supsetneqq}\,\qt^{[\qq]}{=}\,
\qt{\cap}\bar{\qt}.$ 
\end{itemize}
 If $\qt^{{[\qq]}}{\neq}\qt{\cap}\bar{\qt}$ for all integers $\qq{>}0,$ we say that 
 $(\gt_{\R},\qt)$ is \emph{contact degenerate}. 
\end{defn}
We note that we can equivalently use the descending chain 
\begin{equation}
 \label{ideal} 
\begin{cases}
 \at_{\R}^{(0)}\supseteq\at_{\R}^{(1)}\supseteq \cdots \supseteq \at_{\R}^{(\pq-1)}\supseteq
 \at^{(p)}_{\R}\supseteq\at^{(\pq+1)}_{\R}\supseteq \cdots, \quad\text{with}\\
 \at_{\R}^{(0)}=\Ht_{\,\R},\;\;
  \at_{\R}^{(\pq)}{=}\{X\in\at_{\R}^{(\pq-1)}\mid [X,\Ht_{\,\R}]\subseteq\at_{\R}^{(\pq-1)}\},\;\;
 \text{for $\pq{\geq}1$}
\end{cases}
\end{equation}
of \cite{MN05}. This follows from 
\begin{lem}
 With the notation introduced above, we have: 
\begin{enumerate}
 \item For each $\pq{\geq}1,$ $\at_{\R}^{(\pq)}$ is a Lie algebra and, for $\pq{>}1$ an ideal of
 $\at_{\R}^{(1)}$;
 \item Let $\at^{(\pq)}$ be the complexification of $\at^{(\pq)}_{\R}.$ Then 
\begin{equation*}
 \qt^{[\pq]}=\qt{\cap}\bar{\qt}+\at^{(\pq)}\;\;\;\text{for all $\pq{\geq}0.$}
\end{equation*}
\end{enumerate}
\end{lem} 
\begin{proof}
 The first statement is trivial. We can check the second one by recurrence. This is in fact true
 for $\pq{=}0,$ since $\qt{+}\bar{\qt}$ is the complexification of~$\Ht.$ 
\end{proof}
We already considered two descending chains whose length  
defines 
the order of 
contact nondegeneracy. It is in fact convenient
to consider a third one, which is easier to deal with (see \S\ref{sec2.2} below), namely: 
\begin{equation}\label{e1.5}
 \begin{gathered}
 \qt^{\{0\}}\supseteq\qt^{\{1\}}\supseteq\cdots\supseteq\qt^{\{\pq\}}\subseteq\qt^{\{\pq+1\}}\supseteq\cdots \\
 \text{with}\quad \begin{cases} \qt^{\{0\}}=\qt^{(0)}=\{Z\in\qt\mid [Z,\bar{\qt}]\subseteq\qt+\bar{\qt}\},\\
 \qt^{\{\pq\}}=\{Z\in\qt^{\{\pq-1\}}\mid [Z,\qt+\bar{\qt}]\subseteq\qt^{\{\pq-1\}}+\bar{\qt}\},\;\;
 \text{for $\pq{>}0.$}
 \end{cases}
\end{gathered}
\end{equation}
The equivalence is a consequence of 
\begin{prop} With the notation above: 
\begin{itemize}
 \item $\qt^{\{\pq\}}=(\qt\cap\bar{\qt}+\at^{(\pq)})\cap\qt$ for all integers $\pq{\geq}0.$ 
 \item $(\gt_{\R},\qt)$ is contact nondegenerate of order $\qq{\geq}1$ if and only if
 \begin{equation*}\vspace{-19pt}
 \qt^{\{\pq-1\}}\neq\qt^{\{\qq\}}=\qt\cap\bar{\qt}.\end{equation*} \qed
\end{itemize}
 \end{prop}
 Since we obviously have the inclusion 
\begin{equation}
 \qt^{\{\pq\}}\subseteq\qt^{(\pq)}\;\; \forall \pq{\geq}0,
\end{equation}
we obtain 
\begin{prop} 
If the $CR$ algebra $(\gt_{\R},\qt)$ has Levi order $\qq{<}\infty,$ then $(\gt_{\R},\qt)$
has contact  order $\qq'{\leq}\qq.$ A contact degenerate $(\gt_{\R},\qt)$
is also holomorphically degenerate.\qed
\end{prop} 

\section{Orbits of real forms in complex flag manifolds} \label{sec2.2} 
\subsection{\textsc{Complex flag manifolds}}\label{sub2.1}
A complex flag manifold $\sfF$ is
a smooth compact algebraic variety that can be described as the quotient
of a
complex semisimple Lie group $\Sb$ by  a parabolic subgroup $\Qf$;
according to Wolf \cite{Wolf69}, 
a real form $\Sb_{\R}$ of $\Sb$ has finitely many orbits in $\sfF.$
Only one of them, having minimal dimension, is compact. 
With the partial complex structures induced by $\sfF,$
these orbits make a class of homogeneous
$CR$ manifolds that were studied by many authors 
(see e.g. 
\cite{AMN06, AMN08, AMN06b,  BrLo02, fels07, FHW06, GS04, LN2005, MaNa1}).
\paragraph{\textsc{Cross-marked Dynkin diagrams}}
Being connected and simply connected, a
complex flag manifold $\sfF\,{=}\,\Sb/\Qf$ is 
completely described by the Lie pair
$(\st,\qt)$ consisting of the Lie algebras of $\Sb$ and of $\Qf$
and vice versa to any  Lie pair $(\st,\qt)$ of a complex semisimple Lie algebra and
its parabolic subalgebra $\qt$ 
corresponds a unique flag manifold $\sfF.$ 
Therefore  the classification of complex flag manifolds reduces to that
of parabolic subalgebras of semisimple complex Lie algebras.
Parabolic subalgebras  
$\qt$  of $\st$ are classified, modulo automorphisms, by a finite set of parameters.
In fact, 
after fixing any Cartan subalgebra $\hg$ of $\st,$ their equivalence classes
are in one to one correspondence with the subsets of a basis $\Bz$ 
of simple roots 
of the
root system $\Rad$ of $(\st,\hg)$ (see e.g. \cite[Ch.VIII,\S{3.4}]{Bou82}).\par 
We recall that  
the Dynkin diagram $\Delta_{\Bz}$ is a graph with no loops, 
whose nodes are the roots in $\Bz$ and 
in which two nodes may be joined by at most
$3$ edges. 
Each root $\betaup$ in $\Rad$ can be witten in a unique way as a 
notrivial linear combination 
\begin{equation}\label{e1.24}
 \betaup={\sum}_{\alphaup\in\Bz}k_{\betaup,\alphaup}\alphaup,
\end{equation}
with integral coefficients $k_{\betaup,\alphaup}$ which are either all $\geq{0},$
or all $\leq{0}$ and we set 
\begin{equation}
 \supp(\betaup)=\{\alphaup\in\Bz\mid k_{\betaup,\alphaup}\neq{0}\}.
\end{equation}\par 
The parabolic subalgebras
$\qt$ 
are paremetrized, modulo isomorphisms, 
by subsets $\Phi$ of $\Bz$: to a $\Phi\,{\subseteq}\,\Bz$ we associate
\begin{equation} \label{eq1.22}
\begin{cases}
 \Qq_{\Phi}=\{\betaup\in\Rad\mid k_{\betaup,\alphaup}\leq{0},\,\forall\alphaup\in\Phi\},\\
\qt_{\phiup}=\hg\oplus{\sum}_{\betaup\in\Qq_{\Phi}}\st^{\betaup},\;\;\;
\text{with $\st^{\betaup}\,{=}\,\{Z\,{\in}\,\st\,{\mid}\, [H,Z]\,{=}\,\betaup(H)Z,\;\forall H\,{\in}\,\hg\}.$}
\end{cases}
\end{equation}\par 
The set $\Qq_{\Phi}$ is a \emph{parabolic} set of roots, i.e. 
\begin{equation*}
 (\Qq_{\Phi}+\Qq_{\Phi})\cap\Rad\subseteq\Qq_{\Phi}\;\; and \;\; \Qq_{\Phi}\cup(-\Qq_{\Phi})=\Rad.
\end{equation*}
To specify the $\qt_{\Phi}$ of \eqref{eq1.22}
we can cross the nodes corresponding to the roots
in $\Phi.$ In this way each cross-marked Dynkin diagram encodes
a specific complex flag manifold $\sfF_{\Phi}.$ 
\par
\begin{nnott}\label{n1}
Let $\xiup_{\Phi}$ be the linear functional on the linear span of $\Rad$ which equals one
on the roots in $\Phi$ and zero on those in $\Bz{\backslash}\Phi.$ Then 
\begin{equation}
 \Qq_{\Phi}=\{\betaup\in\Rad\mid \xiup_{\Phi}(\betaup)\leq{0}\}
\end{equation}
and we get partitions
\begin{equation} \label{equ2.5}
\left\{
\begin{aligned} \Qq_{\Phi}&=\Qq^{r}_{\Phi}\cup\Qq^{n}_{\Phi},\;\;\;
\Rad =\Qq^{r}_{\Phi}\cup\Qq^{n}_{\Phi}\cup\Qq^{c}_{\Phi},\;\;\text{with}\\
 \Qq^{r}_{\Phi}&=\{\betaup\in\Qq_{\Phi}\mid -\betaup\in\Qq_{\Phi}\}=
 \{\betaup\in\Rad\mid \xiup_{\Phi}(\betaup)=0\},\\
 \Qq^{n}_{\Phi}&=\{\betaup\in\Qq_{\Phi}\mid -\betaup\notin\Qq_{\Phi}\}=
 \{\betaup\in\Rad\mid \xiup_{\Phi}(\betaup)<0\} ,\\
  \Qq^{c}_{\Phi}&=\{\betaup\in\Rad\mid -\betaup\in\Qq^{n}_{\Phi}\}=
 \{\betaup\in\Rad\mid \xiup_{\Phi}(\betaup)>0\}.
 \end{aligned}\right.
\end{equation}

We recall (see e.g. \cite[Ch.VIII,\S{3}]{Bou82}):
\begin{itemize}
 \item $\qt^{r}_{\Phi}=\hg\oplus{\sum}_{\betaup\in\Qq^{r}_{\Phi}}\st^{\betaup}$ is a reductive
 complex Lie algebra;
 \item $\qt^{n}_{\Phi}={\sum}_{\betaup\in\Qq^{n}_{\Phi}}\st^{\betaup}$ is the nilradical of $\qt_{\Phi}$; 
 \item $\qt_{\Phi}=\qt^{r}_{\Phi}\oplus\qt^{n}_{\Phi}$ is the Levi-Chevalley decomposition of $\qt_{\Phi}$;
 \item $\qt_{\Phi}^{c}={\sum}_{\betaup\in\Qq^{c}_{\Phi}}\st^{\betaup}$ is a Lie subalgebra of
 $\st$ consisting of $\ad_{\st}$-nilpotent elements;
 \item 
 $\qt^{\vee}_{\Phi}=\qt^{r}_{\Phi}\oplus\qt^{c}_{\Phi}$ is the
 parabolic Lie subalgebra of $\st$ \emph{opposite of $\qt_{\Phi}$},
 decomposed into the direct sum of its reductive subalgebra $\qt^{r}_{\Phi}$ and
its nilradical $\qt^{c}_{\Phi}.$
\end{itemize}
\end{nnott}
\par
\medskip
\subsection{\textsc{Real forms}} \label{s2.2}
Let us take, as we can, $\Sb$ connected and simply connected.
Then
 real automorphisms of its Lie algebra $\st$ lift to automorphisms of the Lie group $\Sb,$
so that real forms $\Sb_{\R}$ of $\Sb$ are in one-to-one correspondence with
the anti-$\C$-linear involutions $\sigmaup$ of 
$\st.$ We will denote by $\st_{\sigmaup}$ the real Lie subalgebra consisting of
the fixed points  
of $\sigmaup$: it is the Lie algebra of the real form $\Sb_{\sigmaup}$ of fixed points of the
lift $\tilde{\sigmaup}$ of $\sigmaup$ to $\Sb.$ 
Its orbits are $CR$ submanifolds $\sfM_{\Phi,\sigmaup}$  of $\sfF_{\Phi}$
whose $CR$ algebra at the base point $\Qf$ 
is the pair~$(\st_{\sigmaup},\qt_{\Phi}).$ 
 \begin{defn}[cf. {\cite[\S{5}]{AMN06}}] 
 A \emph{parabolic $CR$ algebra} is a pair $(\st_{\sigmaup},\qt_{\Phi})$
 consisting of a real semisimple Lie algebra $\st_{\sigmaup}$ and a  
 parabolic complex Lie subalgebra $\qt_{\Phi}$
 of its complexification $\st.$ 
 We say that $(\st_{\sigmaup},\qt_{\Phi})$ is \emph{minimal} if $\sfM_{\Phi,\sigmaup}$
 is  the minimal orbit 
 in $\sfF_{\Phi}$ of the real form $\Sb_{\sigmaup}$ of $\Sb.$ 
\end{defn}
 When $\st_{\sigmaup}$ is not simple, the corresponding 
 orbits $\sfM_{\Phi,\sigmaup}$ are  $CR$ diffeomorphic to a 
 cartesian product 
 of orbits of simple real Lie groups (see e.g. \cite{AMN06b}).If 
\begin{equation}
 \sfM_{\Phi,\sigmaup}\simeq \sfM_{\Phi_{1},\sigmaup_{1}}\times\cdots\times\sfM_{\Phi_{k},\sigmaup_{k}},
\end{equation}
we call each $\sfM_{\Phi_{i},\sigmaup_{i}}$ a \textit{factor} of $\sfM_{\Phi,\sigmaup}.$
\par
 A simple $\st_{\sigmaup}$ is of the \emph{real type} if also $\st$ is simple;
 otherwise, $\st$ is the direct sum of two complex simple Lie algebras
 $\st',\st''$, which are $\R$-isomorphic to
 $\st_{\sigmaup},$ and we say in this case that $\st_{\sigmaup}$ is of the \emph{complex type}.
\par 
 
\par 
\medskip
To list all the orbits of a real form,
one can use the fact 
that the \textit{isotropy subalgebra} $\st_{\sigmaup}{\cap}\,\qt$ 
contains a Cartan subalgebra $\hg_{\R}$ of $\st_{\sigmaup}$
(see e.g. \cite{AMN06b}). By choosing
$\hg$ equal to its
complexification,
we obtain
on $\Rad$  a conjugation  which is compatible with the one
 defined on $\st$ by its real form $\st_{\sigmaup}$ 
 (and which, for simplicity, we still denote by $\sigmaup$).
 Vice versa, an orthogonal involution $\sigmaup$ of $\Rad$ lifts, 
 although  in general not in a unique way, to a
 conjugation of $\st$. The conjugation on $\st$ depends indeed also on the description
 of which roots in $\Rad_{\,\bullet}^{\sigmaup}\,{=}\,\{\betaup\,{\in}\,\Rad\,{\mid}\,\sigmaup(\betaup)\,{=}\,
 {-}\betaup\}$ are \textit{compact}. This is determined by the choice of a \textit{Cartan involution}
 $\thetaup$ on $\st,$ with $\thetaup(\hg)\,{=}\,\hg$ and $\sigmaup{\circ}\thetaup\,{=}\,\thetaup{\circ}
 \sigmaup$, which induces a map, that we will denote by the same symbol, 
\begin{equation}
 \thetaup:\Rad\ni\alphaup\to{-}\sigmaup(\alphaup)\in\Rad.
\end{equation}
We will write for simplicity $\bar{\alphaup}$ instead of $\sigmaup(\alphaup)$ 
and $\Rad_{\,\bullet}$ for $\Rad_{\,\bullet}^{\sigmaup}$
when this will not
cause confusion. 
We recall that 
$\kt_{\sigmaup}\,{=}\,\{X\,{\in}\,\st_{\sigmaup}\,{\mid}\,\thetaup(X){=}X\}$ is a maximal compact
Lie subalgebra of $\st_{\sigmaup}$ and that we have the \emph{Cartan decomposition}
\begin{equation*}
 \st_{\sigmaup}=\kt_{\sigmaup}\oplus\pt_{\sigmaup},\;\;\;\text{with 
 $\pt_{\sigmaup}=\{X\,{\in}\,\st_{\sigmaup}\,{\mid}\,\thetaup(X){=}{-}X\}$}
\end{equation*}
of $\st_{\sigmaup}$. When $\thetaup(\alphaup)\,{=}\alphaup,$ then 
either $\st^{\alphaup}$ is contained in the complexification $\kt$ of $\kt_{\sigmaup}$
(\textit{compact root}) or in the complexification $\pt$ of $\pt_{\sigmaup}$
(\textit{hermitian root}).

 \par\smallskip
 The subalgebras
$\qt_{\Phi}\,{\cap}\,\bar{\qt}_{\Phi},$ $\qt_{\Phi}$ and $\bar{\qt}_{\Phi}$  turn out to be  
direct sums of $\hg$ and root subspaces
 $\st^{\alphaup}$; in particular   $\qt_{\Phi}\,{\cap}\,\bar{\qt}_{\Phi}$ is the direct sum of $\hg$ and
 the root subspaces $\st^{\alphaup}$ with
$\st^{\alphaup}{+}\st^{\bar{\alphaup}}\,{\subset}\,\qt_{\Phi}.$ 

We note that $\qt_{\Phi}\,{\cap}\,\bar{\qt}_{\Phi}$ is a Lie subalgebra of $\st$ and 
$(\qt_{\Phi}\,{+}\,\bar{\qt}_{\Phi})$ is a $(\qt_{\Phi}\,{\cap}\,\bar{\qt}_{\Phi})$-module. 
\par
Having fixed a base $\Bz$ of simple roots of the root system $\Rad$ associated
to $(\st,\hg),$  
the orbit of the real form is determined by the data of: 
\begin{itemize}
 \item a subset $\Phi$ of $\Bz$ specifying the parabolic subalgebra $\qt_{\Phi}$;
 \item a conjugation $\sigmaup$ of $\Rad$;
 \item a splitting $\Rad^{\sigmaup}_{\,\bullet}\,{=}\,\Rad^{\sigmaup}_{\,\bullet,+}
 {\cup}\Rad^{\sigmaup}_{\,\bullet,-}$ of $\Rad^{\sigmaup}_{\,\bullet}$ into a first set
 consisting of the \textit{compact} 
 and a second of the \textit{hermitian} roots.
\end{itemize}
We point out that different choices of $\sigmaup$ may yield the same $CR$ submanifold
$\sfM_{\Phi,\sigmaup}.$ In particular, we can conjugate $\sigmaup$ by any element of
the subgroup of the Weyl group generated by reflections with respect to roots in
$\Bz{\setminus}\Phi.$
\par\medskip
\subsection{\textsc{Contact nondegeneracy for parabolic $CR$ algebras}} 
Since by definition 
simple Lie algebra have no proper nontrivial ideals,
we obtain 
\begin{prop} A real orbit $\sfM_{\Phi,\sigmaup}$ which is fundamental and 
does not have  a totally complex factor is  contact nondegenerate. 
\end{prop} 
\begin{proof}
 We can indeed reduce to the case where  $(\st_{\sigmaup},\qt_{\Phi})$ is effective
 and  $\st_{\sigmaup}$ is simple, in which the proof is straigthforward. 
\end{proof}
We have the following criterion 
\begin{prop} \label{p2.2}
A real orbit $\sfM_{\Phi,\sigmaup}$ is fundamental iff its
$CR$ algebra $(\st_{\sigmaup},\qt_{\Phi})$ is fundamental. \par 
Let  $(\st_{\sigmaup},\qt_{\Phi})$ be a parabolic $CR$ algebra 
and set 
\begin{equation*}
\Phi^{\sigmaup}_{\circ}=\{\alphaup\in\Phi\mid \sigmaup(\alphaup)\succ 0\}.
 \end{equation*}
If $\Phi^{\sigmaup}_{\circ}\,{=}\,\emptyset,$ then
 $(\st_{\sigmaup},\qt_{\Phi})$ is fundamental.
When $\Phi^{\sigmaup}_{\circ}\,{\neq}\,\emptyset,$ 
 we have 
\begin{itemize}
 \item $(\st_{\sigmaup},\qt_{\Phi})$ is fundamental if and only if 
 $(\st_{\sigmaup},\qt_{\Phi^{\sigmaup}_{\circ}})$ is fundamental;
 \item   $(\st_{\sigmaup},\qt_{\Phi})$ and 
 $(\st_{\sigmaup},\qt_{\Phi^{\sigmaup}_{\circ}})$ are fundamental if and only if 
\begin{equation}\label{eq2.6}
\bar{\Qq}_{\Phi^{\sigmaup}_{\circ}}^{\, c}\cap\Phi^{\sigmaup}_{\circ}=\emptyset.
\end{equation}
\end{itemize}
\end{prop} 
\begin{proof} If $\Phi^{\sigmaup}_{\circ}\,{=}\,\emptyset,$ then
$\Bz\,{\subseteq}\,\Qq_{\Phi}\,{\cup}\,\bar{\Qq}_{\Phi}$ and hence $(\st_{\sigmaup},\qt_{\Phi})$
is trivially fundamental. Let us consider next the case where $\Phi^{\sigmaup}_{\circ}\,{\neq}\,\emptyset.$ 
\par
Since $\Phi^{\sigmaup}_{\circ}\,{\subseteq}\,\Phi,$ we have 
$\qt_{\Phi}\,{\subseteq}\,\qt_{\Phi^{\sigmaup}_{\circ}}$ and therefore 
$(\st_{\sigmaup},\qt_{\Phi^{\sigmaup}_{\circ}})$ is fundamental when $(\st_{\sigmaup},\qt_{\Phi})$
is fundamental. To show the
vice versa, we note that
any Lie subalgebra of $\st$ containing $\qt_{\Phi}$ is of the form
$\qt_{\Psi}$ for some $\Psi\,{\subseteq}\,\Phi.$ If it contains $\qt_{\Phi}{+}\bar{\qt}_{\Phi},$
then $\Psi\,{\subseteq}\,\Phi^{\sigmaup}_{\circ}.$ This proves the first item.
\par
It suffices to prove the second item in the case where 
$\Phi
\,{=}\,\Phi^{\sigmaup}_{\circ}.$ 
Then condition \eqref{eq2.6} is  equivalent to the fact that each $\alphaup\,{\in}\,\Bz$
belongs either to $\Qq_{\Phi}$ or to 
 $\bar{\Qq}_{\Phi}$ and is therefore clearly sufficient for
  $(\st_{\sigmaup},\qt_{\Phi})$ being fundamental. Vice versa, when this condition
  is not satisfied, we can pick  
  $\alphaup\,{\in}\, \bar{\Qq}^{\,c}_{\Phi}{\cap}\Phi.$
  Then $\qt_{\{\alphaup\}}$ is a proper
  parabolic subalgebra of $\st$ containing both $\qt_{\Phi}$
  and $\bar{\qt}_{\Phi}.$ Therefore 
  $\qt_{\Phi}{+}\,\bar{\qt}_{\Phi}$ generates a proper 
  Lie subalgebra of $\st$ and hence  $(\st_{\sigmaup},\qt_{\Phi})$
  is not fundamental.
 This completes the proof.
\end{proof}
\begin{exam} 
Fix $n{\geq}3.$ The cross-marked Dynkin diagram 
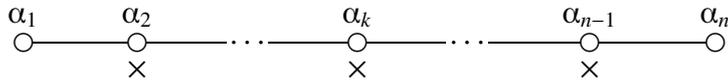
\begin{equation*} 
    \xymatrix@R=-.3pc{
    \alphaup_{1} & \alphaup_{2} && \alphaup_{k}&
    & \alphaup_{n-1}&\alphaup_{n}\\
 \!\!\medcirc\!\!\ar@{-}[r]
&\!\!\medcirc\!\!\ar@{-}[r]
&\! \cdots\! 
&\!\!\medcirc\!\!\ar@{-}[r]\ar@{-}[l]&\!\cdots\!\ar@{-}[r]
&\!\!\medcirc\!\! \ar@{-}[r]
&\!\!\medcirc\!\!
\\
  &\times &&\times&&\times}
\end{equation*}
describes the flag manifold $\sfF_{\Phi}$ of $\SL_{n+1}(\C)$ consisting of
flags $$\ell_{2}\,{\subset}\,\ell_{3}\,{\subset}\,\cdots\,{\subset}\ell_{n-2}\,{\subset}\,\ell_{n-1},$$
where $\ell_{d}$ is a $d$-dimensional linear subspace of $\C^{n+1}.$ Here
$$\Rad\,{=}\,\{{\pm}(e_{i}{-}e_{j})\,{\mid}\, 1{\leq}i{\leq}n{+}1\},\;\;\text{$\alphaup_{i}{=}e_{i}{-}e_{i+1}$
and \;\;$\Phi\,{=}\,\{\alphaup_{i}\,{\mid}\, 2\,{\leq}\,i\,{\leq}\,n{-}1\}$}.
$$
We consider the conjugation $\sigmaup$ defined by 
\begin{equation*}
 \sigmaup(e_{1})={-}e_{n+1},\;\; \sigmaup(e_{i})={-}e_{i},\; \text{for $1{<}i{\leq}n$},\;\; \sigmaup(e_{n+1})={-}e_{1}.
\end{equation*}
Then $(\st_{\sigmaup},\qt_{\Phi})$  is contact nondegenerate of order $[(n{-}1)/2].$ 
It is 
weakly nondegenerate  
for $n\,{=}\,3,4$ and holomorphically degenerate for $n{\geq}5.$ 
\end{exam}

\subsection{\textsc{Conditions for weak  nondegeneracy}}
To discuss weak nondegeneracy, we observe that
the terms of the chain \eqref{e1.1} for $(\st_{\sigmaup},\qt_{\Phi})$ can be described by 
the combinatorics of the root system. We recall that the chain~is 
\begin{gather*}
 \qt_{\Phi}^{(0)}\supseteq\qt_{\Phi}^{(1)}\supseteq\cdots\supseteq\qt_{\Phi}^{(\pq)}
 \supseteq\qt_{\Phi}^{(\pq+1)}\supseteq\cdots \\
\text{with}\;\;\qt_{\Phi}^{(0)}=\qt_{\Phi}\;\;\text{and}\;\; \qt_{\Phi}^{(\pq)}
=\{Z\in\qt_{\Phi}^{(\pq-1)}\mid [Z,\bar{\qt}_{\Phi}]\subseteq
\qt_{\Phi}^{\pq-1}+\bar{\qt}_{\Phi}\}\;\;\text{for $\pq{\geq}1.$}
\end{gather*}
Each $\qt_{\Phi}^{(\pq)}$ in the chain is the direct sum of $\hg$ and root spaces $\st^{\alphaup}.$ 
Let us set 
\begin{equation}
 \Qq_{\Phi}^{\,\pq}=\{\alphaup\in\Rad\mid \st^{\alphaup}\subseteq\qt_{\Phi}^{(\pq)}\}, \;\;\text{so that}\;\;
 \qt_{\Phi}^{(\pq)}=\hg\oplus{\sum}_{\alphaup\in\Qq_{\Phi}^{\pq}}\st^{\alphaup}.
\end{equation}
\par 
With the notation of \S\ref{sub2.1}, we have 
 $\Qq_{\Phi}^{0}=\Qq_{\Phi}$ and 
\begin{equation} 
\begin{cases}
 \Qq_{\Phi}^{1}=\{\alphaup\in\Qq_{\Phi}\mid (\alphaup+\bar{\Qq}_{\Phi})
 \cap\Rad\subseteq \Qq_{\Phi}+\bar{\Qq}_{\Phi}\},\\
 \Qq_{\Phi}^{\,\pq}=\{\alphaup\in\Qq_{\Phi}^{\,\pq-1}\mid 
 (\alphaup+\bar{\Qq}_{\Phi})\cap\Rad\subseteq \Qq_{\Phi}^{\,\pq-1}+\bar{\Qq}_{\Phi}\},\;\;\text{for $\pq{>}1.$}
\end{cases}
\end{equation}
This yields a characterization of weak nondegeneracy in terms of roots:
\begin{prop}
 A necessary and sufficient condition for $(\st_{\sigmaup},\qt_{\Phi})$ being weakly  
 nondegenerate of Levi order $\qq$
 is that $\Qq_{\Phi}^{\,\qq-1}{\neq}\,\Qq_{\Phi}^{\,\qq}\,{=}\,\Qq_{\Phi}\,{\cap}\,\bar{\Qq}_{\Phi}.$ \qed
\end{prop}
\begin{rmk}
 The necessary and sufficient condition for $(\st_{\sigmaup},\qt_{\Phi})$ being
 weakly  nondegenerate 
is that (cf. \cite[Lemma 12.1]{AMN06})
\begin{equation}\label{eq1.29}\begin{cases}
 \forall \betaup\,{\in}\,\Qq_{\Phi}{\backslash}\bar{\Qq}_{\Phi},\;\exists\, 
 k\in\Z_{+},\; \exists\,\alphaup_{1},\hdots,\alphaup_{k}\,{\in}\,\bar{\Qq}_{\Phi}\;\;
 \text{s.t.} \\
  \gammaup_{h}\,{=}\,\betaup{+}{\sum}_{i=1}^{h}\alphaup_{i}\,{\in}\,\Rad,\;\forall 1{\leq} h {\leq} k\;\;
 \gammaup_{k}\notin\Qq_{\Phi}\cup\bar{\Qq}_{\Phi}.\end{cases}
\end{equation}\end{rmk}
\begin{defn} For any root $\betaup\,{\in}\,\Qq_{\Phi}{\backslash}\bar{\Qq}_{\Phi}$ 
we denote by $\qq_{\Phi}^{\sigmaup}(\betaup)$ and call its \emph{Levi order}
 the smallest number $k$ for which \eqref{eq1.29} is valid. We put
 $\qq_{\Phi}^{\sigmaup}(\betaup)\,{=}\,{+}\infty$ when \eqref{eq1.29} is not valid for any positive integer $k.$ 
\end{defn}
\begin{lem} \label{l1.10}
Assume that 
 $\betaup\,{\in}\,\Qq_{\Phi}{\backslash}\bar{\Qq}_{\Phi}$ has finite Levi order  
 $\qq_{\Phi}^{\sigmaup}(\betaup){=}\qq$
and \eqref{eq1.29} is satisfied for a sequence 
$\alphaup_{1},\hdots,\alphaup_{\qq}.$
Then 
\begin{itemize}
\item[$(i)$]
 $\alphaup_{i}\in\bar{\Qq}_{\Phi}{\backslash}\Qq_{\Phi}$ for all $1{\leq}i{\leq}\qq$;
 \item[$(ii)$] $\betaup+{\sum}_{i{\leq}h}\alphaup_{i}\in\Qq_{\Phi}{\backslash}\bar{\Qq}_{\Phi}$
 for all $h{<}\qq$;
 \item[$(iii)$] \eqref{eq1.29} is satisfied by all permutations of 
 $\alphaup_{1},\hdots,\alphaup_{\qq}$;
 \item[$(iv)$] $\alphaup_{i}{+}\alphaup_{j}\,{\notin}\,\Rad$ for all $1{\leq}i{<}j{\leq}\qq.$ 
\end{itemize}
\end{lem} 
\begin{proof} 
Let us first prove $(ii).$  
With the notation in \eqref{eq1.29}, we observe that $\gammaup_{h}\,{\notin}\,\bar{\Qq}_{\Phi}$
for $h{<}\qq,$
because, otherwise, $\gammaup_{\qq}\,{\in}\,\bar{\Qq}_{\Phi}.$
\par
Next we prove $(iii).$ 
Let $\{Z_{\alphaup}\}_{\alphaup\in\Rad}\,{\cup}\,\{H_{i}\,{\in}\,\hg\,{\mid}\,1{\leq}i{\leq}\ell\}$ 
be a Chevalley basis for $(\st,\hg).$ Then \eqref{eq1.29} is equivalent to the fact that 
\begin{equation*} [Z_{\alphaup_{\qq}},
Z_{\alphaup_{\qq-1}},\hdots,Z_{\alphaup_{1}},Z_{\betaup}]
\coloneqq
 [Z_{\alphaup_{\qq}},[Z_{\alphaup_{\qq-1}},[\hdots,[Z_{\alphaup_{1}},Z_{\betaup}]\hdots]]]\notin \qt_{\Phi}
 +\bar{\qt}_{\Phi}.
\end{equation*}
The item  $(iii)$  follows because 
\begin{align*}
[Z_{\alphaup_{\qq}},\hdots, Z_{\alphaup_{i+1}},Z_{\alphaup_{{i}}},\hdots,Z_{\alphaup_{1}},Z_{\betaup}]
-[Z_{\alphaup_{\,\qq}},\hdots, Z_{\alphaup_{i}},Z_{\alphaup_{{i+1}}},\hdots,Z_{\alphaup_{1}},Z_{\betaup}]\qquad\\
=[Z_{\alphaup_{\,\qq}},\hdots, [Z_{\alphaup_{i+1}},Z_{\alphaup_{{i}}}],\hdots,Z_{\alphaup_{1}},Z_{\betaup}]
\end{align*}
and, by the minimality assumption, the right hand side belongs to $\qt_{\Phi}{+}\bar{\qt}_{\Phi}.$ 
\par 
Let us  prove 
$(i)$ by contradiction.
If $\alphaup_{i}\,{\in}\,\Qq_{\Phi}\,{\cap}\,\bar{\Qq}_{\Phi}$ for some $1{\leq}i{\leq}\qq,$
then we could assume by $(iii)$ that it was $\alphaup_{\qq}.$ Then 
\begin{equation*}
 [Z_{\alphaup_{\qq-1}},\hdots,Z_{\alphaup_{1}},Z_{\betaup}]\in\qt_{\Phi}+\bar{\qt}_{\Phi}
 \Longrightarrow 
  [Z_{\alphaup_{\qq}},Z_{\alphaup_{\qq-1}},\hdots,Z_{\alphaup_{1}},Z_{\betaup}] 
  \in \qt_{\Phi}+\bar{\qt}_{\Phi}
\end{equation*}
 yields the contradiction. Also $(iv)$ is an easy consequence of $(iii),$ because
if $\alphaup_{i}{+}\alphaup_{j}$ ($1{\leq}i,j{\leq}\qq$)
is a root, than it would belong to $\bar{\Qq}_{\Phi}\,{\cap}\, \Qq_{\Phi}^{c}$ 
and, by substituting to the two roots $\alphaup_{i},\;\alphaup_{j}$ the single root
$\alphaup_{i}{+}\alphaup_{j}$ we would obtain a sequence satisfying \eqref{eq1.29}
and containing $\qq{-}1$ terms. \par
The proof is complete.
\end{proof}  
\begin{rmk}
Since $\xiup_{\Phi}(\alphaup){\geq}1$ for all $\alphaup\,{\in}\Qq^{\,c}_{\Phi},$ 
 if $\betaup\,{\in}\,\Qq_{\Phi}\,{\cap}\,\bar{\Qq}^{\,c}_{\Phi}$ and
 $\qq^{\sigmaup}_{\Phi}(\betaup){<}{+}\infty,$ then 
\begin{equation}\label{eq2.10}
 \qq^{\sigmaup}_{\Phi}(\betaup)\leq 1-\xiup_{\Phi}(\betaup).
\end{equation}
\end{rmk}
\begin{cor}\label{cor1.11}
 If $\betaup\,{\in}\, \Qq^{\,r}_{\Phi}{\backslash}\bar{\Qq}_{\Phi},$ 
 then its Levi order  is either one or~${+}\infty.$\qed
\end{cor} 
We obtain also a useful criterion of weak nondegeneracy (cf. \cite[Thm.6.4]{AMN06b})
\begin{prop} \label{p2.9}
The parabolic $CR$ algebra $(\st_{\sigmaup},\qt_{\Phi})$
is weakly nondegenerate if and only if 
\begin{equation}\label{e2.12}
 \forall \betaup\,{\in}\,\Qq_{\Phi}\cap\bar{\Qq}^{\,c}_{\Phi}\;\;\exists\, \alphaup\in\bar{\Qq}_{\Phi}
 \cap\Qq^{\,c}_{\Phi}\;\;\text{such that}\;\; \betaup+\alphaup\in\bar{\Qq}_{\Phi}^{\,c}.
\end{equation}
\end{prop} 
\begin{proof} By Lemma~\ref{l1.10} the condition is necessary. To prove that it is also sufficient,
we can argue by contradiction: if we could find
$\betaup\,{\in}\,\Qq\,{\cap}\,\bar{\Qq}^{\,c}_{\Phi}$ with $\qq^{\sigmaup}_{\Phi}(\betaup)\,{=}\,{+}\infty,$
 then by \eqref{e2.12} we could construct an infinite   
 sequence $(\alphaup_{i})_{i\geq{1}}$ in $\bar{\Qq}_{\Phi}\,{\cap}\,\Qq_{\Phi}^{\,c}$
 with 
\begin{equation*}
 \gammaup_{h}=\betaup+{\sum}_{i=1}^{h}\alphaup_{i}\in\Qq_{\Phi}\,{\cap}\,\bar{\Qq}^{\,c}_{\Phi}, \;\;
 \forall h=1,2,\hdots
\end{equation*}
Since $\xiup_{\Phi}(\gammaup_{h})\geq\xiup_{\Phi}(\betaup){+}h$ and $\xiup_{\Phi}$ is bounded,
we get a contradiction. 
\end{proof}
\subsection{\textsc{Levi order of general orbits}}
To discuss Levi order of weakly nondegenerate real orbits $\sfM_{\Phi,\sigmaup}$
in $\sfF_{\Phi}$
by employing Lemma~\ref{l1.10}, we introduce:
\begin{defn}
 If $\betaup\,{\in}\,\Rad,$ we denote by $\qq(\betaup)$ the largest positive integer 
 $\qq$ for which 
\begin{equation} \label{equ1.26}
 \exists \,\alphaup_{1},\hdots,\alphaup_{\qq}\in\Rad \;\;\text{s.t.}\;\;\; 
 \begin{cases}
 \alphaup_{i}{+}\alphaup_{j}\notin\Rad\,{\cup}\,\{0\},\; 
 \forall 1{\leq}i,j{\leq}\qq,\\
 \gammaup_{i_{1},\hdots,i_{h}}=\betaup+\alphaup_{i_{1}}+\cdots+\alphaup_{i_{h}}\in\Rad, \\
 \quad \text{for all distinct $i_{1},\hdots,i_{h}$ in $\{1,\hdots,\qq\}.$}
 \end{cases}
\end{equation}
\end{defn}
\begin{prop} \label{pr1.11}
Let $\betaup\,{\in}\,\Rad$ belong to a simple root system containing more than two elements. Then 
$\qq(\betaup){\leq}4$ and, if $\qq(\betaup){=}4$ and $(\alphaup_{1},\alphaup_{2},\alphaup_{3},\alphaup_{4})$
is a sequence satisfying \eqref{equ1.26}, then 
\begin{equation}
 \betaup+\alphaup_{1}+\alphaup_{2}+\alphaup_{3}+\alphaup_{4}=-\betaup.
\end{equation}More precisely we obtain:
\begin{equation} 
\begin{cases}
\qq(\betaup){=}1, & \text{if $\betaup$ belongs to a root system of type $\mathrm{A_{2}}$};\\[-2pt]
 &\text{or is a long root of a root system of type $\mathrm{B_{2}}$};\\
 \qq(\betaup){=}2, & \text{if $\betaup$ 
 belongs to a root system of type $\mathrm{A_{{\geq}3},C},$}
 \\[-2pt]
 &\text{or is a short root of a  system of type $\mathrm{B_{2},G}$};\\
 \qq(\betaup){=}3, & \text{if $\betaup$ is a short root of a root system of type $\mathrm{B_{\geq{3}},F}$;}
 \\[2pt]
 \qq(\betaup){=}4, & \text{if $\betaup$ belongs to a root system of type $\mathrm{D,E},$}
\\[-2pt]
 &\text{or is a long root of a root system of type $\mathrm{B_{\geq{3}},F,G}$.}
\end{cases}
\end{equation}
\end{prop} 
\begin{proof} For short we will call \emph{admissible} a sequence $(\alphaup_{i})$ for which
\eqref{equ1.26} is valid.
Let us set 
\begin{equation*}
\Rad^{\!\text{add}}(\betaup)=\{\alphaup\in\Rad\mid \betaup{+}\alphaup\in\Rad\}.
\end{equation*}
We consider the different cases using
for root systems the notation of~\cite{Bou68}.\par\smallskip
\paragraph{Type $\mathrm{A}$} We have $\Rad\,{=}\,\{\pm(e_{i}-e_{j})\,{\mid}\, 
1{\leq}i{<}j{\leq}n\}$ where $e_{1},\hdots,e_{n}$ is an orthonormal basis of
 $\R^{n}.$ We can take
$\betaup=e_{2}{-}e_{1}.$ Then 
\begin{equation*} \tag{$*A$}
\Rad^{\!\text{add}}(e_{2}{-}e_{1})=\{e_{1}-e_{i}\,{\mid}\, i{>}2\}\cup\{e_{i}-e_{2}\,{\mid}\, i>2\}. 
\end{equation*}
An admissible 
sequence $(\alphaup_{i})$ can contain at most one element from each of the two sets
in the right hand side of  $(*A).$ \par 
If $n\,{=}\,3,$ then $\Rad^{\!\text{add}}(\betaup)\,{=}\,\{e_{3}{-}e_{2},\, e_{1}{-}e_{3}\}$
contains two elements, whose sum is still a root and therefore $\qq(\betaup)\,{=}\,1.$ 
\par 
If $n{>}3,$ then the only possible choice is that of a couple of roots
$e_{i}{-}e_{2},\, e_{1}{-}e_{j}$ with $3\,{\leq}i{\neq}j{\leq}n$ and hence $\qq(\betaup)\,{=}\,2.$ 
\par\smallskip
\paragraph{Type $\mathrm{B}$} We have $\Rad\,{=}\,\{{\pm}e_{i}{\pm}e_{j}\,{\mid}\, 1{\leq}i{<}j{\leq}n\}
{\cup}\{{\pm}e_{i}\,{\mid}\, 1{\leq}i{\leq}n\}, $ for an orthonormal basis $e_{1},\hdots,e_{n}$ of $\R^{n}$
($n{\geq}2$).\par
If $\betaup$ is a short root, we can take $\betaup\,{=}\,{-}e_{1}.$ Then 
\begin{equation*}
 \tag{$*B$}\Rad^{\!\text{add}}({-}e_{1})=\{{\pm}e_{i}\mid 2\leq{i}\leq{n}\}\cup\{e_{1}{\pm}e_{j}\mid 2\leq{j}\leq{n}\}.
\end{equation*}
An admissible sequence contains at most one root from the first 
and two from the second 
set in the right hand side of
$(*B).$ Thus $\qq(-e_{1}){\leq}3.$ 
The sequence $e_{1}{-}e_{2},\,e_{1}{+}e_{2}$ satisfies \eqref{equ1.26} and therefore
$\qq(-e_{1}){\geq}2.$ \par
We have equality if $n{=}2,$ because in this case 
$\Rad^{\!\text{add}}(-e_{1})\,{=}\{{\pm}{e}_{2},\, e_{1}{\pm}e_{2}\}$ and the maximal admissible sequences
are then $(e_{2}),$ $({-}e_{2}),$ $(e_{1}{+}e_{2},\, e_{1}{-}e_{2}).$
\par 
 If $n{>}2$ the admissible sequence
 $$
 (e_{1}{+}e_{2},\, e_{1}{-}e_{2},e_{3})$$ 
 shows that 
$\qq(-e_{1}){=}3.$ All admissible maximal sequences are of this form.
 \par\smallskip
If $\betaup$ is a long root, we can assume that $\betaup\,{=}\,-e_{1}{-}e_{2}.$ Then 
\begin{equation*} \tag{$*\!*\!{B}$}
\Rad^{\!\text{add}}({-}e_{1}{-}e_{2}){=}
 \{e_{1},e_{2}\}\cup \{e_{1}{\pm}e_{j}\,{\mid}\,\; j{>}2\}
 \cup \{e_{2}{\pm}e_{j}\,{\mid}\, \; j{>}2\}.
\end{equation*}
An admissible sequence contains at most two equal terms 
from the first and two from 
each of the second
and third on the right hand side of $({*\!*\!{B}}).$ Moreover, if one term is taken from the first,
we can take at most one from each one of the other two. 
This implies that $\qq(-e_{1}{-}e_{2}){\leq}4$
and 
in fact $\qq({-}e_{1}{-}e_{2}){=}4,$ with maximal sequences  isomorphic to one of 
\begin{align*}
& e_{1}+e_{3},\; e_{1}-e_{3},\; e_{2}+e_{4},\; e_{2}-e_{4},\\
& e_{1},\; e_{1},\; e_{1}{-}e_{3},\; e_{1}{+}e_{3},
\end{align*}
which, summed up to $({-}e_{1}{-}e_{2}),$ gives $e_{1}{+}e_{2}.$ 
\par\medskip
\paragraph{Type $\mathrm{C}$} 
 We can take $\Rad\,{=}\,\{{\pm}e_{i}{\pm}e_{j}\,{\mid}\, 1{\leq}i{<}j{\leq}n\}
{\cup}\{{\pm}2e_{i}\,{\mid}\, 1{\leq}i{\leq}n\}, $ for an orthonormal basis $e_{1},\hdots,e_{n}$ of $\R^{n}$
($n{\geq}3$).\par
If $\betaup$ is a short root, we can assume that $\betaup{=}({-}e_{1}{-}e_{2}).$ Then 
\begin{equation*}
 \tag{$*C$}\Rad^{\!\text{add}}({-}e_{1}{-}e_{2})=\{2e_{1},2e_{2}\}\cup\{e_{1}{\pm}e_{j}\,{\mid}\, j\,{\geq}\,3\}
 \cup\{e_{2}{\pm}e_{j}\,{\mid}\,  j\,{\geq}\,3\}.
\end{equation*}
An admissible sequence may contain both roots of the first, but at most one root from each the
second and third sets on the right hand side of $(*C).$ Moreover, a term in one of the last two
forbids the corresponding term in the first one. This yields $\qq({-}e_{1}{-}e_{2}){=}2,$ with
maximal sequences isomorphic to (the third one should be omitted if $n{=}3$)
\begin{align*}
 (2e_{1},\; 2e_{2}),\;\;\; (2e_{1},\, e_{2}{+}e_{3}),\;\; (e_{1}{+}e_{3},\, e_{2}{+}e_{4})
\end{align*}
\par\smallskip
If $\betaup$ is a long root, we can assume that $\betaup\,{=}\,{-}2e_{1}.$ Then 
\begin{equation*}
 \tag{*\! *C}\Rad^{\!\text{add}}({-}2e_{1})=\{e_{1}{\pm}e_{i}\mid i>1\}.
\end{equation*}
We note that $\qq(-2e_{1}){\leq}4.$ 
We cannot take in an admissible sequence both the element $e_{1}{+}e_{i}$ and
$e_{1}{-}e_{i},$ because they add up to the root $2e_{i}.$ 
Hence in fact $\qq({-}2e_{1}){=}2,$ with maximal sequence isomorphic to 
\begin{align*}
& e_{1}{+}e_{2},\;  e_{1}{+}e_{3}.
\end{align*} 
 \par\medskip
 \paragraph{Type $\mathrm{D}$} We can take $\Rad\,{=}\,\{{\pm}e_{i}{\pm}e_{j}\,{\mid}\, 1{\leq}i{<}j
 {\leq}n\},$ where $e_{1},\hdots,e_{n}$ is an orthonormal  basis of $\R^{n}$ ($n{\geq}4$).\par
 We can assume that $\betaup={-}e_{1}{-}e_{2}.$ We have 
\begin{equation*}
 \tag{$*D$}\Rad^{\!\text{add}}({-}e_{1}{-}e_{2})=\{e_{1}{\pm}e_{j}\,{\mid}\, j\,{\geq}\,3\}
 \cup\{e_{2}{\pm}e_{j}\,{\mid}\,  j\,{\geq}\,3\}.
\end{equation*}
An admissible sequence contains at most two elements from each set in the right hand side
of $(*D).$ Therefore $\qq({-}e_{1}{-}e_{2}){\leq}4$ and in fact we have equality with maximal
admissibe sequences isomorphic to 
\begin{equation*}
 e_{1}{+}e_{3},\;  e_{1}{-}e_{3},\; e_{2}{+}e_{4},\; e_{2}{-}e_{4},
\end{equation*}
which, summed up to $({-}e_{1}{-}e_{2}),$ give $e_{1}{+}e_{2}.$
\par\medskip  
\paragraph{Type $\mathrm{E}$} Since the root systems $\mathrm{E}_{6}$ and $\mathrm{E}_{7}$
can be considered as subsystems of $\mathrm{E}_{8},$ we will restrain to this case. We consider,
for an orthonormal basis $e_{1},\hdots,e_{8}$ of $\R^{8},$  
\begin{equation*}
 \Rad=\{{\pm}{e}_{i}{\pm}e_{j}\mid 1{\leq}i{<}j{\leq}8\}\cup
 \left.\left\{\tfrac{1}{2}{\sum}_{i=1}^{8}(-1)^{h_{i}}e_{i}\,\right| \, h_{i}\in\Z,\; {\sum}_{i=1}^{8}h_{i}\in 2\Z\right\}.
\end{equation*}
We can take $\betaup\,{=}\,({-}e_{1}{-}e_{2}).$ Then 
\begin{equation*}
 \tag{$*E$}\left\{
  \begin{aligned}
\Rad^{\!\text{add}}({-}e_{1}{-}e_{2})=\{{e}_{1}{\pm}e_{i}\mid 3{\leq}i{\leq}8\}\cup
 \{{e}_{2}{\pm}e_{i}\mid 3{\leq}i{\leq}8\}\qquad\qquad\\
 \cup
 \left.\left\{\tfrac{1}{2}\left(
 e_{1}{+}e_{2}{+}{\sum}_{i=3}^{8}(-1)^{h_{i}}e_{i}\right)\,\right| \, h_{i}\in\Z,\; {\sum}_{i=3}^{8}h_{i}\in 2\Z\right\}
 \end{aligned}\right.
\end{equation*}
An admissible sequence may contain at most two roots from each set on the right hand
side of $(*E)$ and no more than four terms. Clearly we can take the maximal sequence 
\begin{equation*}
 e_{1}{+}e_{3},\; e_{1}{-}e_{3},\; e_{2}{+}e_{4},\; e_{2}{-}e_{4},
\end{equation*}
showing that $\qq({-}e_{1}{-}e_{2})\,{=}\,4.$ Moreover, any admissible sequence containing
four terms sums up to $({-}e_{1}{-}e_{2})$ to yield $e_{1}{+}e_{2}.$ 
\par\medskip  
\paragraph{Type $\mathrm{F}$} For an orthonormal basis $e_{1},e_{2},e_{3},e_{4}$ of $\R^{4}$
we take 
\begin{equation*}
 \Rad=\{{\pm}e_{i}\mid 1{\leq}i{\leq}4\}\cup\{{\pm}e_{1}{\pm}e_{j}\mid 1{\leq}i{<}j{\leq}4\}
 \cup\{\tfrac{1}{2}({\pm}e_{1}{\pm}e_{2}{\pm}e_{3}{\pm}e_{4})\}.
 \end{equation*}
 \par If $\betaup$ is a short root, we can take $\betaup{=}{-}e_{1}.$ Then 
\begin{equation*}
 \tag{$*F$}
\Rad^{\!\text{add}}(-e_{1})=\{{\pm}e_{i}\mid 2{\leq}i{\leq}4\}\cup\{e_{1}{\pm}e_{i}\mid 2{\leq}i{\leq}4\}\cup
 \{\tfrac{1}{2}(e_{1}{\pm}e_{2}{\pm}e_{3}{\pm}e_{4})\}.
\end{equation*}
To build an an admissible sequence
we can take at most one element  
from the first, two from the second and  from the
third set in the right hand side of $(*F).$ Indeed two roots of the form
$\tfrac{1}{2}({\pm}e_{1}{\pm}e_{2}{\pm}e_{3}{\pm}e_{4})$ do not add up to a root if
and only if they differ by only one sign. Moreover, no root can be taken from the first
if one is taken from the last set. These considerations imply that $\qq({-}e_{1}){\leq}3$
and in fact equality holds, as $({-}e_{1})$ is contained in a subsystem of type
$\mathrm{B}_{3}.$ \par \smallskip
If $\betaup$ is a long root, we can assume $\betaup{=}({-}e_{1}{-}e_{2}).$  We have
\begin{equation*}
 \tag{$*\! * \!{F}$} \left\{ \begin{aligned}
\Rad^{\!\text{add}}({-}e_{1}{-}e_{2})=\{e_{1},e_{2}\}\cup
\{e_{1}{\pm}e_{i}\mid 3{\leq}i{\leq}4\}\quad\\
\cup\{e_{2}{\pm}e_{i}\mid 3{\leq}i{\leq}4\}
\cup 
\{\tfrac{1}{2}(e_{1}{+}e_{2}{\pm}e_{3}{\pm}e_{4})\}.
\end{aligned} \right. 
\end{equation*}
We note that the sum of four terms of $\Rad^{\!\text{add}}({-}e_{1}{-}e_{2})$
is a linear combination $\betaup +k_{1}e_{1}{+}k_{2}e_{2}{+}k_{3}e_{3}{+}k_{4}e_{4}$ with
$k_{1}{+}k_{2}{\geq}2$ and therefore, if they form an admissible sequence,
is equal to $e_{1}{+}e_{2}.$ Since $\Rad$ contains subsystems of type $\mathrm{B}_{3},$
there are indeed admissible sequences with four elements. \par\medskip
\paragraph{Type $\mathrm{G}$} For an orthonormal basis $e_{1},e_{2},e_{3}$ of $\R^{3}$ we set 
\begin{equation*}
 \Rad=\{{\pm}(e_{i}-e_{j})\mid 1{\leq}i{<}j{\leq}3\}\cup\{{\pm}(2e_{i}-e_{j}-e_{k})\mid \{i,j,k\}=\{1,2,3\}\}.
\end{equation*}
We consider firs the case of a short root. We can take $\betaup=e_{2}{-}e_{1}.$ Then 
\begin{equation*}\tag{$*{G}$}
\Rad^{\!\text{add}}(e_{2}{-}e_{1})=\{e_{3}{-}e_{2},\; e_{1}{-}e_{3}\}\cup\{ 2e_{1}{-}e_{2}{-}e_{3},\;
 e_{1}{+}e_{3}{-}2e_{2}\}.
\end{equation*}
Maximal admissible sequences have a root from the first and one from the second set, hence
$\qq(e_{2}{-}e_{1})\,{=}\,2$ and, moreover, summed up to $e_{2}{-}e_{1},$ give $e_{1}{-}e_{2}.$ 
\par
As a long root we take $\betaup=(e_{2}{+}e_{3}{-}2e_{1}).$ Then 
\begin{equation*}
 \tag{$*\! *\! {G}$}\Rad^{\!\text{add}}(e_{2}{+}e_{3}{-}2e_{1})=\{e_{1}{-}e_{2},\; e_{1}{-}e_{3}\}
 \cup \{e_{1}{+}e_{2}{-}2e_{3},\; e_{1}{+}e_{3}{-}2e_{2}\}.
\end{equation*}
One checks that in this case $\qq(e_{2}{+}e_{3}{-}2e_{1})\,{=}\,4,$ with 
a maximal admissible
sequence 
\begin{equation*}
 e_{1}{-}e_{2},\;  e_{1}{-}e_{2},\;  e_{1}{-}e_{2},\; e_{1}{+}e_{2}{-}2e_{3}
\end{equation*}
which indeed sums up to the opposite root $2e_{1}{-}e_{2}{-}e_{3}.$ 
\par
The proof is complete.
\end{proof}
As an easy consequence we obtain:
\begin{thm}\label{thm1.11} Let $\sfM_{\Phi,\sigmaup}$ be a real orbit which is fundamental
and weakly nondegenerate. 
 Then its Levi order  is less or equal to $3.$
\end{thm} 
\begin{proof}
 This is a consequence of Prop.\ref{pr1.11} and the fact that, if $\betaup$ does not belong 
to $\bar{\Qq}_{\Phi},$ then 
 ${-}\betaup\in\bar{\Qq}_{\Phi}$ because $\bar{\Qq}_{\Phi}$ is a parabolic
 set of roots. 
\end{proof}
\begin{exam} (\cite[\S{7}]{fels07}) \label{ex2.12}
Let $n$ be an integer $\geq{3}$ and fix  
a symmetric 
$\C$-bilinear 
form $\bil$ on 
$\C^{2n+1}\! .$ The Lie algebra of
the group of $\C$-linear transformations of $\C^{2n+1}$
that keep $\bil$ invariant is a simple complex
Lie algebra $\ot_{2n+1}(\C)$
of type $\mathrm{B}_{n},$ with root system \begin{equation*}
 \Rad=\{{\pm}e_{i}\mid{1}\leq{i}\leq{n}\}\cup\{{\pm}e_{i}{\pm}e_{j}\mid 1{\leq}i{<}j{\leq}n\}
\end{equation*}
for an orthonormal basis $e_1,\hdots,e_n$ of $\R^n.$
Fix $k$ with $1{<}k{<}n.$ The cross-marked Dynkin diagram 
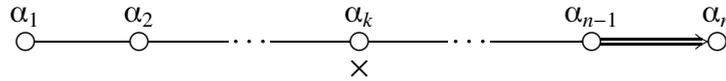
\begin{equation*} 
    \xymatrix@R=-.3pc{
    \alphaup_{1} & \alphaup_{2} && \alphaup_{k}&
    & \alphaup_{n-1}&\alphaup_{n}\\
 \!\!\medcirc\!\!\ar@{-}[r]
&\!\!\medcirc\!\!\ar@{-}[r]
&\! \cdots\! 
&\!\!\medcirc\!\!\ar@{-}[r]\ar@{-}[l]&\!\cdots\!\ar@{-}[r]
&\!\!\medcirc\!\! \ar@{=>}[r]
&\!\!\medcirc\!\!
\\
  & &&\times&&}
\end{equation*}
represents
the grassmannian of totally $\bil$-isotropic
$k$-planes in $\C^{2n+1}.$ 
Here $\alphaup_{i}{=}e_{i}{-}e_{i+1}$ for $1{\leq}i{<}n$ and $\alphaup_{n}\,{=}\,e_{n}.$
We have 
$\Phi\,{=}\,\{\alphaup_{k}\}$ and 
\begin{equation*}
 \xiup_{\Phi}(e_{i})= 
\begin{cases}
 1, & \text{if $1{\leq}i{\leq}k,$}\\
 0, & \text{if $k{<}i{\leq}n.$}
\end{cases}
\end{equation*}
Real forms are obtained by fixing a conjugation $\sigmaup$ on $\C^{2n+1}.$ Then 
\begin{equation*}
\bil_{\sigmaup}(\vq,\wq)=\bil(\vq,\sigmaup(\wq)),\;\;\forall \vq,\wq\in\C^{2n+1}
\end{equation*}
is hermitian symmetric and nondegenerate, of signature $(p,q)$ for a pair of nonnegative integers
with $p{+}q{=}2n{+}1.$ The Lie algebra of the group of $\C$-linear transformations which
keep fixed both $\bil$ and $\bil_{\sigmaup}$ is a real form $\st_{\sigmaup}$ of 
$\st{\simeq}\ot_{2n+1}(\C),$ which is isomorphic to the real simple Lie algebra $\ot(p,q).$\par
We define a conjugation $\sigmaup$ on $\Rad$ by 
\begin{equation*}
 \sigmaup(e_{1})=e_{n},\quad \sigmaup(e_{i})={-}e_{i}, \;\text{if $1\,{<}\,i\,{<}\,n,$}
 \quad \sigmaup(e_{n})=e_{1}.
\end{equation*}
According to the number of compact roots between $e_{2},\hdots,e_{n-1},$ this conjugation
corresponds to any of the Lie algebras $\ot(p,2n{+}1-p)$ with $1{\leq}p{\leq}2n.$ 
The orbit $\sfM_{\Phi,\sigmaup}$ consists of 
$\bil_{\sigmaup}$-isotropic
$k$-spaces $\ell_{k}$ 
with $\dim(\ell_{k}\cap\sigmaup(\ell_{k})){=}k{-}1.$ 
By \eqref{e1.13},
since $\qt_{\Phi}$ is maximal,
 if  
the parabolic $CR$ algebra $(\st_{\sigmaup},\qt_{\Phi})$ is fundamental and not totally
complex, then it is also weakly nondegenerate. 
To compute its Levi order  we observe that 
\begin{equation*} \begin{cases}
 \Qq_{\Phi}^{c}\cap\bar{\Qq}_{\Phi}^{c}=\{e_{1}+e_{n}\},\\
\Qq_{\Phi}^{c}\cap \bar{\Qq}_{\Phi}=\{e_{i}\mid 1{\leq}i{\leq}k\}\cup
 \{e_{i}{\pm}e_{j}\mid 1{\leq}i{\leq}k{<}j{\leq}n\},\\
 \begin{aligned}
 \Qq_{\Phi}\cap\bar{\Qq}_{\Phi}^{c}=\{-e_{i}\mid 2{\leq}i{\leq}k\}\cup\{e_{n}\}\cup\{e_{n}{\pm}e_{j}\mid
 k{<}j{<}n\}\qquad\quad\\ 
 \cup  \{{-}e_{i}{\pm}e_{j}\mid 2{\leq}i{\leq}k{<}j{<}n\}
 \cup \{{\pm}e_{1}{-}e_{i}\mid 2{\leq}i{\leq}k\}
 \end{aligned}
 \end{cases}
\end{equation*}
The fact that $\Qq_{\Phi}^{c}\cap\bar{\Qq}_{\Phi}^{c}{\neq}\emptyset$ 
shows that $(\st_{\sigmaup},\qt_{\Phi})$ is not totally complex.\par 
Since $e_{n}\,{\in}\,{\Qq}_{\Phi}\cap\bar{\Qq}_{\Phi}^{c}$ 
and $e_{1}\,{\in}\,  \bar{\Qq}_{\Phi}\cap\Qq_{\Phi}^{c}$
add up to $e_{1}{+}e_{n},$ the $CR$ algebra $(\st_{\sigmaup},\qt_{\Phi})$ is fundamental
and therefore, as we noticed above, weakly nondegenerate. \par 
The roots $\betaup_{i}\,{=}\,-(e_{1}{+}e_{i}),$ for $2{\leq}i{\leq}k$ belong to
 ${\Qq}_{\Phi}\cap\bar{\Qq}_{\Phi}^{c}$ and have $\xiup_{\Phi}(\betaup_{i})\,{=}\,{-}2.$
Since $\xiup_{\Phi}(e_{1}{+}e_{2})\,{=}\,1$  
and 
$\xiup_{\Phi}(\alphaup)\,{\leq}\,1$ for all $\alphaup\,{\in}\,\bar{\Qq}_{\Phi}\cap\Qq_{\Phi}^{c},$
no chain \eqref{eq1.29} that added up to $\betaup_{i}$ yields $e_{1}{+}e_{n}$ contains
less than three elements. By Thm.\ref{thm1.11} this shows that
$(\st_{\sigmaup},\qt_{\Phi})$ has Levi  order $3.$
We have indeed 
\begin{equation*}
  ({-}e_{1}{-}e_{i})+e_{1}+e_{1}+(e_{i}{+}e_{n})=e_{1}{+}e_{n}.
\end{equation*}
\end{exam}

\begin{exam}
Consider a simple complex Lie algebra $\st$ of type $\mathrm{D}_{4}.$ 
Its root system is described, by using an orthonormal basis 
$e_{1},e_{2},e_{3},e_{4}$ of $\R^{4},$~by
\begin{equation*}
 \Rad=\{{\pm}e_{i}{\pm}e_{j}\mid 1{\leq}i{<}j{\leq}4\}.
\end{equation*}
Consider the complex flag manifod $\sfF_{\Phi}$ corresponding to the cross-marked Dynkin
diagram
  
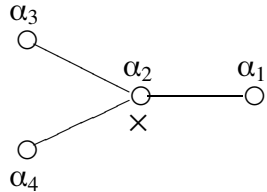
\begin{equation*}
 \xymatrix@R=-.3pc{
 \alphaup_3\\
\!\!
\medcirc\!\!\! \ar@{-}[rdd] \\
&\alphaup_{2} &\alphaup_{1}\\
 & \!\!\medcirc\!\!\! \ar@{-}[r]  & \!\!\medcirc\!\!\! 
\\
&\times\\
\!\!
\medcirc\!\!\! \ar@{-}[ruu]
\\
 \alphaup_{4}}
\end{equation*}
with $\alphaup_{1}=e_{1}-e_{2},$ $\alphaup_{2}=e_{2}-e_{3},$ $\alphaup_{3}=e_{3}-e_{4},$ 
$\alphaup_{4}=e_{3}+e_{4}$ and $\Phi\,{=}\,\{\alphaup_{2}\}.$  \par 
 It is the grassmannian of projective lines contained in
the nondegenerate quadric complex hypersurface in $\CP^{7}.$ The grading functional is 
\begin{equation*}
 \xiup_{\Phi}(e_{i})= 
\begin{cases}
 1, & i=1,2,\\
 0, & i=3,4.
\end{cases}
\end{equation*}
Take the conjugation 
\begin{equation*}
 \sigmaup(e_{1})=e_{4},\;\;  \sigmaup(e_{2})=-e_{2},\;\; \sigmaup(e_{3})=-e_{3},\;\; 
 \sigmaup(e_{4})=e_{1}.\;\;
 \end{equation*}
We have 
\begin{equation*}
 \Qq_{\Phi}^{c}\cap\bar{\Qq}^{c}=\{e_{1}{+}e_{4}\}.
\end{equation*}
This shows that $(\st_{\sigmaup},\qt_{\Phi})$ is not totally complex and therefore,
since $\qt_{\Phi}$ is maximal parabolic, this $CR$ algebra is weakly nondegenerate iff
it is fundamental. We have 
\begin{equation*} 
\begin{cases}
 {\Qq}_{\Phi}^{\,c}\,{\cap}\,\bar{\Qq}_{\Phi}=\{e_{1}{-}e_{3}, e_{1}{-}e_{4},
 e_{2}{-}e_{3}, e_{2}{-}e_{4},
 e_{3}{-}e_{4},  e_{1}{+}e_{2},e_{1}{+}e_{3}, e_{2}{+}e_{3}, e_{2}{+}e_{4}\}\\
 \begin{aligned}
 \Qq_{\Phi}\cap\bar{\Qq}^{c}_{\Phi}=\{e_{3}{+}e_{4},\; e_{4}{-}e_{1},\; e_{3}{-}e_{2}, \; {-}e_{1}{-}e_{2},
 {-}e_{1}{-}e_{3}, e_{4}{-}e_{2}, e_{4}{-}e_{3}, \;\quad\qquad \\
 {-}e_{2}{-}e_{3}, e_{1}{-}e_{2}\}.
 \end{aligned}
\end{cases}
\end{equation*}
Note that $e_{3}{+}e_{4}\,{\in}\,\Qq_{\Phi}\cap\bar{\Qq}^{c}$ and $e_{1}{-}e_{3}\,{\in}\,
\bar{\Qq}\,{\cap}\,\Qq^{c}$ sum up to $e_{1}{+}e_{4}.$ This shows that $(\st_{\sigmaup},\qt_{\Phi})$
is not totally complex and fundamental. Since $\qt_{\Phi}$ is maximal parabolic, this
implies that $(\st_{\sigmaup},\qt_{\Phi})$ is weakly nonedegenerate.
\par 
The root $\betaup{=}{-}e_{1}{-}e_{2}$ belongs to $\Qq_{\Phi}\,{\cap}\,
\bar{\Qq}^{c}_{\Phi}$ and $\xiup_{\Phi}(\betaup)\,{=}\,{-}2.$ Since all roots $\alphaup$
in $\bar{\Qq}_{\Phi}\,{\cap}\,\Qq^{c}_{\Phi}$ 
distinct from $e_{1}{+}e_{2}$ have $\xi_{\betaup}(\alphaup){=}1,$ 
a sequence satisfying \eqref{eq1.29} and summing up with $\betaup$
to $\{e_{1}{+}e_{4}\}$ contains at least $3$ elements. 
By Thm.\ref{thm1.11} this shows that
$(\st_{\sigmaup},\qt_{\Phi})$ has Levi  order $3.$
An admissible sequence for ${-}e_{1}{-}e_{2}$ is 
$
 (e_{1}{-}e_{3},\; e_{1}{+}e_{3},\; e_{2}{+}e_{4}).
$
 \end{exam}
\begin{exam} Consider a semisimple complex Lie algebra $\st$ of type $\mathrm{G}_{2}.$ 
Having fixed a Cartan subalgebra, we can write its root system in the form 
\begin{equation*}
 \Rad=\{{\pm}(e_{i}-e_{j})\mid 1{\leq}i{<}j{\leq}3\}\cup\{\pm(2e_{\sigmaup_{1}}-e_{\sigmaup_{2}}-e_{\sigmaup_{3}})
 \mid \sigmaup\in\Sb_{3}, \; \sigmaup_{2}<\sigmaup_{3}\},
\end{equation*}
for an orthonormal basis $e_{1},e_{2},e_{3}$ of $\R^{3}.$ 
We consider the complex flag manifold $\sfF_{\Phi}$ 
corresponding to the cross-marked Dynkin diagram
\begin{equation*}
  \xymatrix@R=-.3pc{  
\alphaup_2 & \alphaup_1  \\
\!\!\medcirc\!\!\!
\ar@3{->}[r]&\!\!\medcirc\!\! \\
\times}\quad \qquad\begin{matrix} 
\qquad\quad \\
\alphaup_1{=}e_{1}-e_{2}\;\; \alphaup_2{=}2e_{2}-e_{1}-e_{3}.
\end{matrix}
\end{equation*}
It corresponds to the grading functional $\xiup_{\Phi}$ with 
\begin{equation*}
 \xiup_{\Phi}(e_{1})=1,\;\;\xiup_{\Phi}(e_{2})=1,\;\;\xiup_{\Phi}(e_{3})=0.
\end{equation*}
We consider the conjugation defined by 
\begin{equation*}
 \sigmaup(e_{1})=e_{3},\;\;\sigmaup(e_{2})=e_{2},\;\;  \sigmaup(e_{3})=e_{1}.
\end{equation*}
Then 
\begin{equation*}
 \Qq^{c}_{\Phi}\cap\bar{\Qq}^{c}_{\Phi}=\{2e_{2}-e_{1}-e_{3}\}.
\end{equation*}
Since $\qt_{\Phi}$ is maximal and  $\Qq^{c}_{\Phi}\,{\cap}\,\bar{\Qq}^{c}_{\Phi}{\neq}\emptyset,$
then is sufficient to check that 
$(\st_{\sigmaup},\qt_{\Phi})$ is weakly nondegenerate to find that it is also fundamental. 
\par The root $\betaup\,{=}\, 2e_{3}{-}e_{1}{-}e_{2}$ belongs to $\Qq_{\Phi}\,{\cap}\,\bar{\Qq}^{c}_{\Phi}.$ 
We have 
\begin{equation*}
\Qq^{\,c}_{\Phi}\cap \bar{\Qq}_{\Phi}=\{e_{1}-e_{3},\; e_{2}-e_{3},\; 2e_{1}-e_{2}-e_{3}\}.
\end{equation*}
Since $\xiup_{\Phi}$ equals one on every root of $\Qq^{\,c}_{\Phi}\,{\cap}\, \bar{\Qq}_{\Phi},$
a sequence satisfying \eqref{eq1.29} has at least three roots. We find indeed that 
\begin{equation*}
 e_{2}-e_{3},\; e_{2}-e_{3},\; e_{2}-e_{3}
\end{equation*}
is a sequence with the desired properties, proving that $(\st_{\sigmaup},\qt_{\Phi})$
is fundamental and has Levi  order three. 
\end{exam}
\par\medskip

\subsection{\textsc{Levi order of  orbits of the minimal type}}
Weak nondegeneracy for minimal orbits  
was characterized
in \cite[Thm.11.5]{AMN06} by using their description in terms 
of cross-marked Satake diagrams (see e.g. 
 \cite{Ara62,Satake1960}).
\par  
Let 
$\hg_{\R}$ be a maximally vectorial Cartan subalgebra
of $\st_{\sigmaup},$ $\hg$ its complexification and
$\Rad$ the root system  of $(\st,\hg).$
Then all roots in $\Rad_{\,\bullet}$ are compact. 
We can select a basis $\Bz$  such that the conjugate of any positive 
noncompact root stays positive. This condition defines 
an involution $\epi\,{:}\,\Bz{\to}\Bz,$ which keeps fixed
the elements of 
$\Bz_{\,\bullet}\,{=}\,\Bz\,{\cap}\,\Rad_{\,\bullet}$ 
and    
such that, for nonnegative $n_{\alphaup,\betaup}{\in}\Z,$  
\begin{equation} \label{satake}
\begin{cases}
\bar{\alphaup}=-\alphaup, & \forall\alphaup\in\Bz_{\bullet},\\
\bar{\alphaup}=\epi(\alphaup)+{\sum}_{\betaup\in\Bz_{\bullet}}n_{\alphaup,\betaup}\betaup,
&\forall \alphaup\in\Bz{\backslash}\Bz_{\bullet}.
\end{cases}
\end{equation}
\par 
The Satake  diagram $\Sigma_{\Bz}$ is obtained from $\Delta_{\Bz}$
 by painting black the roots in
$\Bz_{\bullet}$ and joining by an 
 arch the pairs of distinct simple roots $\alphaup_{1},\alphaup_{2}$ 
with $\epi(\alphaup_{1})=\alphaup_{2}.$ \par 
Minimal orbits correspond to \emph{cross-marked Satake diagrams}: they
are associated to parabolic $\qt_{\Phi}$ for which all roots in 
$\Qq_{\Phi}^{\,c}{\cap}\,\bar{\Qq}_{\Phi}^{\,n}$
are compact.  \par \smallskip 
Let us 
drop the assumption that
$\hg_{\R}$ is maximally vectorial.
The map $\thetaup\,{:}\,\alphaup\,{\mapsto}\,{-}\bar{\alphaup}$ 
induced on $\Rad$ by the Cartan involution 
(see \S\ref{s2.2})
acts on  
 $\Qq_{\Phi}^{\,c}\cap\bar{\Qq}^{\,n}_{\Phi}$, which is therefore the union of 
  its fixed points, which are roots in $\Rad_{\,\bullet}$,
 and of pairs $(\alphaup,{-}\bar{\alphaup})$ of distinct roots.
 \begin{defn} We say that the $CR$ algebra $(\st_{\sigmaup},\qt_{\Phi})$ and the 
 corresponding $CR$ manifold $\sfM_{\Phi,\sigmaup}$ are \emph{of the minimal type}
 if the roots in $\Qq_{\Phi}^{\,c}\cap\bar{\Qq}^{\,n}_{\Phi}$ are fixed by the Cartan involution, i.e. if 
\begin{equation}
\Qq^{\,c}_{\Phi}\cap\bar{\Qq}^{\,n}_{\Phi}\subseteq\Rad_{\,\bullet}\\
\label{eq2.15}\end{equation}
\end{defn}
\begin{lem} For a parabolic $CR$ algebra 
$(\st_{\sigmaup},\qt_{\Phi})$ the following are equivalent
to the fact that it is of the minimal type:
\begin{align}
\label{eq2.15a}
&\xiup_{\Phi}(\bar{\betaup})\geq{0},\;\;\forall \betaup\in\Qq^{\,c}_{\Phi}{\setminus}\Rad_{\,\bullet};
\\ 
\label{eq2.15b}
 &\xiup_{\Phi}(\bar{\betaup})=0,\;\;\forall \betaup\in(\Qq^{\,c}_{\Phi}\cap\bar{\Qq}_{\Phi}){\backslash}
 \Rad_{\,\bullet}.
\end{align}
\end{lem}
\begin{proof} 
 \eqref{eq2.15a} is equivalent to \eqref{eq2.15b}.
Indeed,
since $\xiup_{\Phi}(\bar{\betaup})\,{\leq}\,0$ for all $\betaup\,{\in}\bar{\Qq}_{\Phi},$ 
clearly \eqref{eq2.15a} is a consequence of \eqref{eq2.15b}. The two are equivalent because
$\xiup_{\Phi}(\bar{\betaup}){>}0$ for $\betaup\,{\in}\,\bar{\Qq}^{\,c}_{\Phi}$ and
$\Rad\,{=}\,\bar{\Qq}_{\Phi}\,{\cup}\,\bar{\Qq}^{\,c}_{\Phi}.$ 
The equivalence of \eqref{eq2.15} with \eqref{eq2.15a} reduces to  
the observation that 
the elements of $\Qq_{\Phi}^{\,c}$ on which $\xiup_{\Phi}{\circ}\,\sigmaup$ is negative 
make the set $\Qq^{\,c}_{\Phi}\cap\bar{\Qq}^{\,n}_{\Phi}$.
 \end{proof}

\begin{exam} Keep the notation of Examp.\ref{ex2.12}. The cross-marked
Dynkin diagram of $\mathrm{B}_{3}$ 
 \begin{equation*} 
    \xymatrix@R=-.3pc{
    \alphaup_{1} & \alphaup_{2} & \alphaup_{3}\\
 \!\!\medcirc\!\!\ar@{-}[r]
&\!\!\medcirc\!\! \ar@{=>}[r]
&\!\!\medcirc\!\! 
\\
\times&&\times}
\end{equation*}
corresponds to 
\begin{equation*}
 \Phi=\{\alphaup_{1},\alphaup_{3}\},\;\; \xiup_{\Phi}(e_{i})= 
\begin{cases}
 2, & i{=}1,\\
 1, & i=2,3.
\end{cases}
\end{equation*}
Consider the conjugation 
\begin{equation*}
 \sigmaup(e_{1})=e_{2},\;\;\sigmaup(e_{2})=e_{1},\;\;\sigmaup(e_{3})\,{=}\,{-}e_{3}.
\end{equation*}
Since $\Phi\,{\subset}\,\Rad_{\,\bullet},$ by Prop.\ref{p2.2} the $CR$ algebra
$(\st_{\sigmaup},\qt_{\Phi})$ is fundamental.
We have 
\begin{align*}
 &\Qq^{\,c}_{\Phi}\cap\bar{\Qq}^{\,c}_{\Phi}=\{e_{1},\,e_{2},\, e_{1}{+}e_{2},\, e_{1}{-}e_{3},\,
 e_{2}{+}e_{3}\},\\
 &\Qq^{\,c}_{\Phi}\cap\bar{\Qq}_{\Phi}=\{e_{3},\,e_{1}{+}e_{3},\,  e_{1}{-}e_{2}\},\\
  &\Qq_{\Phi}\cap\bar{\Qq}^{\,c}_{\Phi}= \{{-}e_{3},\,e_{2}{-}e_{3},\,  e_{2}{-}e_{1}\}.
\end{align*}
This $(\st_{\sigmaup},\qt_{\Phi})$ is of the minimal type, because 
$$\Qq^{c}\,{=}\,(\Qq^{\,c}_{\Phi}\cap\bar{\Qq}^{\,c}_{\Phi})\,{\cup}\,\{e_{1}{-}e_{2},\,e_{3}\}
\,{\subset}\,(\Qq^{\,c}_{\Phi}\cap\bar{\Qq}^{\,c}_{\Phi})\,{\cup}\,\Rad_{\,\bullet}.$$
However, $(\st_{\sigmaup},\qt_{\Phi})$ is not the $CR$ algebra of the minimal orbit 
of a real form of $\SO_{7}(\C)$ in $\sfF_{\Phi},$ because, although $\alphaup_{1},\alphaup_{3}
\,{\in}\,\Rad_{\,\bullet}$ and $\bar{\alphaup}_{2}{=}\alphaup_{1}{+}\alphaup_{2}{+}2\alphaup_{3}
\,{\succ}\,0,$  showing that 
the basis $\alphaup_{1},\alphaup_{2},\alphaup_{3}$ defines an
$S$-chamber according to \cite{AMN06b},
the diagram obtained by
blackening the nodes $\alphaup_{1},\alphaup_{3}$ is not Satake.
The equalities
\begin{equation*}\begin{cases}
 ({-}e_{3}{+}(e_{1}{+}e_{3}){=}e_{3}\,{\in}\, \Qq^{\,c}_{\Phi}\cap\bar{\Qq}^{\,c}_{\Phi},\\
 (e_{2}{-}e_{3}){+}e_{3}{=}e_{2}\,{\in}\,
 \Qq^{\,c}_{\Phi}\cap\bar{\Qq}^{\,c}_{\Phi},\\
 (e_{2}{-}e_{1}){+}(e_{1}{+}e_{3}){=}e_{2}{+}e_{3}\,{\in}\,
 \Qq^{\,c}_{\Phi}\cap\bar{\Qq}^{\,c}_{\Phi},
 \end{cases}
\end{equation*}
show that $(\st_{\sigmaup},\qt_{\Phi})$ is  weakly nondegenerate.\par
We can choose real forms $\SO(2,5)$ or $\SO(3,4)$ compatible with the complex
symmetric bilinear form $\bil$  used to define $\SO_{7}(\C).$ Then $\sfM_{\Phi,\sigmaup}$
consists of pairs $(\ell_{1}\,{\subset}\,\ell_{3})$ with a $\bil_{\sigmaup}$-isotropic $\ell_{1}$
with $\ell_{1}{\cap}\bar{\ell}_{1}\,{=}\,\{0\}$ and an $\ell_{3}$ on which the restriction
of $\bil_{\sigmaup}$ has rank $1.$ 
\end{exam}

\begin{thm}\label{t2.16} 
A real orbit $\sfM_{\Phi,\sigmaup}$ of the minimal type is either
holomorphically degenerate or has Levi order less or equal two. 
\end{thm} 
\begin{proof}
 Let $(\st_{\sigmaup},\qt_{\Phi})$ be a parabolic $CR$ algebra of the minimal type. Keeping the
 notation used throughout the section, we note that \eqref{eq2.15} can be rewritten in the
 form 
\begin{equation*}\tag{$*$}
 \xiup_{\Phi}(\betaup)\geq{0},\;\;\forall \betaup\,{\in}\bar{\Qq}^{\,c}_{\Phi}\backslash\Rad_{\,\bullet}.
\end{equation*}
Let $\betaup\,{\in}\,\Qq_{\Phi}\,{\cap}\,\bar{\Qq}_{\Phi}^{\,c}.$ If $\betaup\,{\notin}\,\Rad_{\,\bullet},$
then $\xiup_{\Phi}(\betaup)\,{=}\,0$ by ($*$) and hence, by Cor.\ref{cor1.11},
$\qq^{\sigmaup}_{\Phi}(\betaup)$ is either $1$ or ${+}\infty.$ \par
Let us consider now the case where 
$\qq^{\sigmaup}_{\Phi}(\betaup)$ is an integer $q{>}1.$ 
Then $\betaup\,{\in}\,\Rad_{\,\bullet}.$ 
Let $(\alphaup_{1},\hdots,\alphaup_{q})$ 
be a sequence satisfying \eqref{eq1.29} and thus the conditions in Lemma~\ref{l1.10}.
Since $\Qq^{\,c}_{\Phi}\,{\cap}\,\bar{\Qq}^{\,c}_{\Phi}\,\cap\,\Rad_{\,\bullet}\,{=}\,\emptyset,$
there is at least one root $\alphaup_{i}$ which does not belong to $\Rad_{\,\bullet}.$
By the Lemma we can assume it is $\alphaup_{1}.$ Then $\betaup{+}\alphaup_{1}$
belongs to $(\Qq_{\Phi}\,{\cap}\,\bar{\Qq}^{\,c}_{\Phi}){\backslash}\Rad_{\,\bullet}$ and therefore,
by the first part of the proof, $\qq^{\sigmaup}_{\Phi}(\betaup{+}\alphaup_{1}){=}1.$
This implies that $q{=}2.$ The proof is complete. 
\end{proof}
\begin{lem} 
 The parabolic $CR$ algebra $(\st_{\sigmaup},\qt_{\Phi})$ 
 of a minimal orbit is of the minimal type.
\end{lem}
\begin{proof}
 Suppose that $\Phi$ is the set of crossed roots in a  
 cross-marked Satake diagram. 
 Since all roots $\betaup$ 
 in $\Qq^{\,c}_{\Phi}$ are positive, 
by  \eqref{satake}, if $\betaup\,{\in}\,{\Qq}^{\,c}_{\Phi}{\setminus}\Rad_{\,\bullet},$ 
then its conjugate 
$\bar{\betaup}$ is positive, and hence has $\xiup_{\Phi}(\betaup){\geq}0.$
This shows that \eqref{eq2.15} is valid, i.e. that $\sfM_{\Phi,\sigmaup}$ is of
the minimal type.
\end{proof}
\begin{cor} A minimal orbit $\sfM_{\Phi,\sigmaup}$ is either holomorphically degenerate or 
has
Levi order   less or equal to two. \qed
\end{cor}
 \begin{exam} \label{ex2.5}
 Consider the $CR$ algebra described by the cross-marked Satake diagram
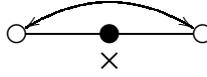
\begin{equation*} 
    \xymatrix@R=-.3pc{\!\!\medcirc\!\!\ar@{-}[r]\ar@{<->}@/^1pc/[rr]
&\!\!\medbullet\!\! \ar@{-}[r]&\!\!\medcirc\!\!\\
  &\times}
\end{equation*}
It is associated to the minimal orbit $\sfM_{\Phi,\sigmaup}$ of $\SU(1,3)$ is the
Grassmannian  of isotropic two-planes of $\C^{4}$ for 
a hermitian symmetric form of signature~$(1,3).$  \par
Here $\st\,{\simeq}\,\slt_{4}(\C),$ $\Rad\,{=}\,\{{\pm}(e_{i}-e_{j})\,{\mid}\, 1{\leq}i{<}j{\leq}4\},$
$\Bz\,{=}\,\{\{e_{1}{-}e_{2},\,e_{2}{-}e_{3},\,e_{3}{-}e_{4}\} $ for an orthonormal basis
$e_{1},e_{2},e_{3},e_{4}$ of $\R^{4},$ $\Phi\,{=}\,\{e_{2}{-}e_{3}\}$, 
\begin{equation*}
 \xiup(e_{i})= 
\begin{cases}
 1, & i{=}1,2,\\
 0, & i{=}3,4,
\end{cases} \quad \
\begin{cases}
 \sigmaup(e_{1})={-}e_{4},\;\; \sigmaup(e_{2})={-}e_{2},\\
 \sigmaup(e_{3})={-}e_{3},\;\; \sigmaup(e_{4})={-}e_{1}.
\end{cases}
\end{equation*}
We obtain 
\begin{align*}
 &\Qq^{c}_{\Phi}\cap\bar{\Qq}^{c}_{\Phi}=\{e_{1}-e_{4}\},\\
 &\bar{\Qq}_{\Phi}\cap\Qq^{c}_{\Phi}=\{e_{1}-e_{3},e_{2}-e_{3},e_{2}-e_{4}\},\\
 &\Qq_{\Phi}\cap\bar{\Qq}^{c}_{\Phi}=\{e_{3}-e_{4},e_{3}-e_{2}, e_{1}-e_{2}\}.
\end{align*}
Since $\Qq^{c}_{\Phi}\cap\bar{\Qq}^{c}_{\Phi}$ is nonempty, 
$e_{1}{-}e_{4}\,{=}\,(e_{3}{-}e_{4}){+}(e_{1}{-}e_{3})$
and $\qt_{\Phi}$ is maximal, we obtain that $(\st_{\sigmaup},\qt_{\Phi})$
is fundamental and weakly nondegenerate. Since $\xiup_{\Phi}(e_{3}{-}e_{2}){=}{-}1$
and $\xiup_{\Phi}$ is $1$ on all 
the elements of $\bar{\Qq}_{\Phi}\cap\Qq^{c}_{\Phi},$ the Levi order 
is at least, and thus equal,  by Thm.\ref{t2.16}, to $2$. We have in fact
\begin{equation*}
 e_{1}-e_{4}=(e_{3}-e_{2})+(e_{1}{-}e_{3})+(e_{2}{-}e_{4}),\;\;
 e_{1}-e_{4}=(e_{1}-e_{2})+(e_{2}{-}e_{4}).
\end{equation*}
\end{exam}
\begin{exam}
  Consider the $CR$ algebra described by the cross-marked Satake diagram 
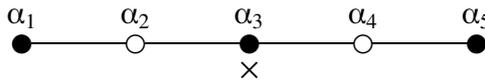
\begin{equation*} 
    \xymatrix@R=-.3pc{
\alphaup_{1}&\alphaup_{2}&\alphaup_{3}&\alphaup_{4}&\alphaup_{5}\\    
    \!\!\medbullet\!\! \ar@{-}[r]&\!\!\medcirc\!\!\ar@{-}[r]
&\!\!\medbullet\!\! \ar@{-}[r]&\!\!\medcirc\!\!\ar@{-}[r]
&\!\!\medbullet\!\! 
\\
  &&\times}
\end{equation*}
Here $\Rad\,{=}\,\{{\pm}(e_{i}{-}e_{j})\,{\mid}\,1{\leq}i{<}j{\leq}6\}$, $\alphaup_{i}\,{=}\,e_{i}{-}e_{i+1},$
$\Phi\,{=}\,\{\alphaup_{3}\},$ 
\begin{equation*} \xiup_{\Phi}(e_{i})= 
\begin{cases}
 1, & \text{for $1{\leq}i{\leq}3,$}\\
 0, & \text{for $4{\leq}i{\leq}6,$}
\end{cases}\qquad 
 \sigmaup(e_{i})= 
\begin{cases}
 e_{i+1}, &\text{if $i$ is odd,}\\
 e_{i-1},&\text{if $i$ is even.}
\end{cases}
\end{equation*}
It corresponds to the $CR$ algebra $(\st_{\sigmaup},\qt_{\Phi}),$ with $\st_{\sigmaup}\,{\simeq}\,\slt_{3}(\Hb),$ 
of the grassmannian $\sfM_{\Phi,\sigmaup}$ 
 of $3$-planes of $\C^{6}\,{\simeq}\,\Hb^{3}$ 
containing a quaternionic line. We have 
\begin{align*}
& \Qq^{\,c}_{\Phi}\cap\bar{\Qq}^{\,c}_{\Phi}=\{e_{1}{-}e_{5},\, e_{1}{-}e_{6},\, e_{2}{-}e_{5},\, e_{2}{-}e_{6}\},\\
&\Qq^{\,c}_{\Phi}\cap\bar{\Qq}_{\Phi}=\{e_{1}{-}e_{4},\, e_{2}{-}e_{4},\, e_{3}{-}e_{4},\, e_{3}{-}e_{5}, 
e_{3}{-}e_{6}\},\\
&\Qq_{\Phi}\cap\bar{\Qq}_{\Phi}^{\,c}=\{e_{2}{-}e_{3},\, e_{1}{-}e_{3},\, e_{4}{-}e_{3},\, e_{4}{-}e_{6}, 
e_{4}{-}e_{5}\}
\end{align*}
Since $\qt_{\Phi}$ is maximal, it suffices to note that $(e_{4}{-}e_{5}){+}(e_{1}{-}e_{4}){=}
e_{1}{-}e_{5}\,{\in}\,\Qq^{\,c}_{\Phi}\cap\bar{\Qq}^{\,c}_{\Phi}$ to conclude that
$(\st_{\sigmaup},\qt_{\Phi})$ is fundamental and weakly nondegenerate. \par
We have $\Qq_{\Phi}\,{\cap}\,\bar{\Qq}_{\Phi}^{\,c}\,{\cap}\,\Rad_{\,\bullet}\,{=}\,\{e_{4}{-}e_{3}\}.$ 
Since both 
\begin{equation*}
 (e_{4}{-}e_{3}){+}(e_{1}{-}e_{4})=e_{1}{-}e_{3}\in\Qq_{\Phi}\cap\bar{\Qq}_{\Phi}^{\,c} ,\;\; 
  (e_{4}{-}e_{3}){+}(e_{2}{-}e_{4})=e_{2}{-}e_{3}\in\Qq_{\Phi}\cap\bar{\Qq}_{\Phi}^{\,c}
\end{equation*}
we get $\qq^{\sigmaup}_{\Phi}(e_{4}{-}e_{3})\,{=}\,2,$ showing that the Levi order  of 
$(\st_{\sigmaup},\qt_{\Phi})$ equals two.
\end{exam} 
\begin{exam} \label{ex2.70}
 The $CR$ algebra  
 described by the cross-marked Satake diagram 
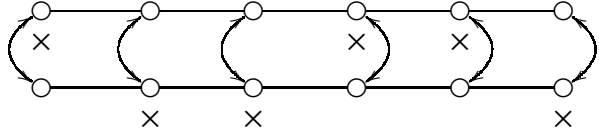
\begin{equation*} 
    \xymatrix@R=-.15pc{
     \!\!\medcirc\!\!\ar@{-}[r]\ar@{<->}@/_1pc/[ddd]
&\!\!\medcirc\!\! \ar@{-}[r]\ar@{<->}@/_1pc/[ddd]
&\!\!\medcirc\!\!\ar@{<->}@/_1pc/[ddd]\ar@{-}[r]
&\!\!\medcirc\!\!\ar@{<->}@/^1pc/[ddd]\ar@{-}[r]
&\!\!\medcirc\!\! \ar@{-}[r]\ar@{<->}@/^1pc/[ddd]
&\!\!\medcirc\!\! \ar@{<->}@/^1pc/[ddd]\\
\times&&&\times&\times\\
\quad\\
\!\!\medcirc\!\!\ar@{-}[r]
&\!\!\medcirc\!\! \ar@{-}[r]
&\!\!\medcirc\!\!\ar@{-}[r]
&\!\!\medcirc\!\!\ar@{-}[r]
&\!\!\medcirc\!\!\ar@{-}[r]
&\!\!\medcirc\!\! \\
  &\times&\times&&&\times}
\end{equation*}
corresponding to $(\st_{\sigmaup},\qt_{\Phi}),$ with $\st_{\sigmaup}\,{\simeq}\,\slt_{7}(\C),$ 
is fundamental and weakly nondegenerate. This can be proved e.g. by applying the criteria
in \cite{AMN06}. 
Since $\Rad_{\,\bullet}{=}\emptyset,$
its Levi order 
is one.  
\end{exam}
\begin{exam} Keep the notation of Example \ref{ex2.12} and consider the cross-marked
Dynkin diagram of $\mathrm{B}_{3}$ 
 \begin{equation*} 
    \xymatrix@R=-.3pc{
    \alphaup_{1} & \alphaup_{2} & \alphaup_{3}\\
 \!\!\medcirc\!\!\ar@{-}[r]
&\!\!\medcirc\!\! \ar@{=>}[r]
&\!\!\medcirc\!\! 
\\
\times}
\end{equation*}
corresponding to 
\begin{equation*}
 \Phi=\{\alphaup_{1}\},\;\; \xiup_{\Phi}(e_{i})= 
\begin{cases}
 1, & i{=}1,\\
 0, & i=2,3.
\end{cases}
\end{equation*}
Consider the conjugation 
\begin{equation*}
 \sigmaup(e_{1})={-}e_{2},\;\;\sigmaup(e_{2})={-}e_{1},\;\;\sigmaup(e_{3})\,{=}\,e_{3}.
\end{equation*}
Then, for the corresponding $CR$ algebra $(\st_{\sigmaup},\qt_{\Phi})$,  
we have 
\begin{align*}
 &\Qq^{\,c}_{\Phi}\cap\bar{\Qq}^{\,c}_{\Phi}=\{e_{1}{-}e_{2}\},\\
 &\Qq^{\,c}_{\Phi}\cap\bar{\Qq}_{\Phi}=\{e_{1},\,e_{1}{+}e_{2},\, e_{1}{+}e_{3},\, e_{1}{-}e_{3}\},\\
  &\Qq_{\Phi}\cap\bar{\Qq}^{\,c}_{\Phi}= \{{-}e_{2},\,{-}e_{1}{-}e_{2},\, -e_{2}{+}e_{3},\, {-}e_{2}{-}e_{3}\}.
\end{align*}
Since $\xiup_{\Phi}(\gammaup){=}1$ for all $\gammaup\,{\in}\,\Qq^{\,c}_{\Phi}\cap\bar{\Qq}_{\Phi},$
 and $\xiup_{\Phi}(\gammaup){\geq}{-}1$ for all $\gammaup\,{\in}\,\Qq^{\,c}_{\Phi}\cap\bar{\Qq}_{\Phi},$
 the Levi order  of $(\st_{\sigmaup},\qt_{\Phi})$ is two. 
 Then $(\st_{\sigmaup},\qt_{\Phi})$ is a 
 $CR$ algebra is of the minimal type, although is not the $CR$ algebra of a minimal orbit.
\end{exam}
\begin{rmk} It was observed in \cite{AMN06b} that a parabolic $CR$ algebra 
$(\st_{\sigmaup},\qt_{\Phi})$ can always be described by using a base $\Bz$ associated to
an \textit{$S$-chamber}: this means one with  $\bar{\alphaup}\,{\succ}\,0$ for all
$\alphaup\,{\in}\Bz{\setminus}(\Phi{\cup}\Bz_{\,\bullet}).$ 
 The condition of being of \textit{the minimal type} translates for this choice of $\Bz$  
 into the fact that
$\bar{\alphaup}\,{\succ}\,0$ also for the elements in $\Phi{\setminus}\Bz_{\,\bullet}.$ 
The real dimension of $\sfM_{\Phi,\sigmaup}$ is the difference 
$\dim_{\C}(\st){-}\dim_{\C}(\qt_{\Phi}{\cap}\bar{\qt}_{\Phi}),$ i.e.
$\#\Rad{-}\,\#(\Qq_{\Phi}{\cap}\bar{\Qq}_{\Phi})$. Thus, in case $\Phi$ contains a root
$\alphaup\,{\notin}\,\Bz_{\,\bullet}$ with $\bar{\alphaup}\,{\prec}\,0,$ the symmetry with
respect to $\alphaup$ yields a new basis $\Bz'$ that, with the crosses in the same
positions, describes a new real orbit whose dimension is smaller by one unit.  
Then, parametrizing the real orbits 
that we can describe, after having made a  
fixed choice of $\hg_{\R}$, 
by using the Weyl chambers of $\Rad,$ those
of the minimal type are a sort of \textit{local minima} with respect to dimension. 
One has to be cautious because, unless $\hg_{\R}$ is maximally vectorial, 
there can be  several inequivalent choices of $\Bz$ 
such that $\bar{\alphaup}\,{\succ}\,0$
for all $\alphaup\,{\in}\,\Bz{\setminus}\Bz_{\,\bullet}$ that we can look at as
yielding different \textit{local minima} for the dimension of a class of real orbits.
\end{rmk}
\subsection{\textsc{Further examples}} 
We already showed that there are weakly nondegenerate $CR$ algebras $(\st_{\sigmaup},\qt_{\Phi})$
of Levi order $3,$ which, by Thm.\ref{t2.16}, are not of the minimal type. In 
this subsection we exhibit  
examples of  weakly nondegenerate parabolic $CR$ algebras which are not of
the minimal type and have Levi orders $1,2.$ 
\begin{exam} Consider $\slt_{3}(\C)$ as a simple real Lie algebra. Its complexification is the
direct sum of two copies of $\slt_{3}(\C).$ Its root system can be described, after fixing
orthogonal basis $e_{1},e_{2},e_{3},e_{4}$ and $e'_{1},e'_{2},e'_{3},e'_{4}$ of two copies
of $\R^4,$ by 
\begin{equation*}
 \Rad=\{{\pm}(e_i-e_j)\mid 1{\leq}i{<}j{\leq}4\}\cup\{{\pm}(e'_i-e'_j)\mid 1{\leq}i{<}j{\leq}4\}.
\end{equation*}
Let us consider the cross-marked Dynkin diagram 
\begin{equation*} 
    \xymatrix@R=-.3pc{ \alphaup_1 &\alphaup_2&\alphaup_3\\
    \!\!\medcirc\!\!\ar@{-}[r]
&\!\!\medcirc\!\! \ar@{-}[r]&\!\!\medcirc\!\! \\
\times &&\times\\
\alphaup'_1 &\alphaup'_2&\alphaup'_3
\\
    \!\!\medcirc\!\!\ar@{-}[r]
&\!\!\medcirc\!\! \ar@{-}[r]&\!\!\medcirc\!\! \\
\times &&\times}
\end{equation*} 
where $\alphaup_i{=}(e_{i}{-}e_{i+1})$ and $\alphaup'_i{=}(e'_{i}{-}e'_{i+1}),$
with $\Phi{=}\{\alphaup_1,\alphaup_3\}{\cup}\{\alphaup'_1,\alphaup'_3\}.$
\par
Let us fix the conjugation 
\begin{equation*}
 \sigmaup(e_{i})= 
\begin{cases}
 e'_{i+1}, & i=1,3,\\
 e'_{i-1}, & i=2,4,
\end{cases}\quad \sigmaup(e'_{i})= 
\begin{cases}
 e_{i+1}, & i=1,3,\\
 e_{i-1}, & i=2,4,
\end{cases}
\end{equation*}
Then 
\begin{align*}
 \Qq^{\,c}_{\Phi}\cap\bar{\Qq}^{\,c}_{\Phi}&=\{e_{1}{-}e_{3},\,e_{2}{-}e_{4}\}\cup\{e'_{1}{-}e'_{3},\,e'_{2}{-}e'_{4}\},\\
  \Qq^{\,c}_{\Phi}\cap\bar{\Qq}_{\Phi}&=\{e_{1}{-}e_{2},\, e_{1}{-}e_{4},\,e_{3}{-}e_{4}\}
  \cup \{e'_{1}{-}e'_{2},\, e'_{1}{-}e'_{4},\,e'_{3}{-}e'_{4}\},\\
   \Qq_{\Phi}\cap\bar{\Qq}^{\,c}_{\Phi}&=\{e_{2}{-}e_{1},\, e_{2}{-}e_{3},\, e_{4}{-}e_{3}\}
   \cup \{e'_{2}{-}e'_{1},\, e'_{2}{-}e'_{3},\, e'_{4}{-}e'_{3}\}.
\end{align*}
Since all roots in $\Phi$ have a negative conjugate, the parabolic $CR$ algebra 
$(\st_{\sigmaup},\qt_{\Phi})$ is fundamental. It is not of the minimal type because $\Rad_{\,\bullet}
\,{=}\,\emptyset$ and 
$$\Qq^{\,n}_{\Phi}{\cap}\bar{\Qq}^{\,c}_{\Phi}\,{=}\,\{e_{2}{-}e_{1},\, e_{4}{-}e_{3}\}
\cup \{e'_{2}{-}e'_{1},\, e'_{4}{-}e'_{3}\}\neq\emptyset.$$
Let us check that $(\st_{\sigmaup},\qt_{\Phi})$ has Levi order $1.$ We get indeed 
\begin{align*}
& (e_{2}{-}e_{1})+(e_{1}{-}e_{4})=(e_{2}{-}e_{4}), & (e'_{2}{-}e'_{1})+(e'_{1}{-}e'_{4})=(e'_{2}{-}e'_{4}),\\
& (e_{2}{-}e_{3})+(e_{3}{-}e_{4})=(e_{2}{-}e_{4}), & (e'_{2}{-}e'_{3})+(e'_{3}{-}e'_{4})=(e'_{2}{-}e'_{4}),\\
& (e_{2}{-}e_{3})+(e_{1}{-}e_{2})=(e_{1}{-}e_{3}) ,& (e'_{2}{-}e'_{3})+(e'_{1}{-}e'_{2})=(e'_{1}{-}e'_{3}),\\
& (e_{4}{-}e_{3})+(e_{1}{-}e_{4})=(e_{1}{-}e_{3}) ,& (e'_{4}{-}e'_{3})+(e'_{1}{-}e'_{3})=(e'_{1}{-}e'_{3}).
\end{align*}
This also shows that $(\st_{\sigmaup},\qt_{\Phi})$ is weakly nondegnerate. 
The orbit $\sfM_{\Phi,\sigmaup}$ is a $CR$ manifold of $CR$ dimension $6$ and $CR$ codimension
$4.$ Its points are quadruples $(\ell_{1},\ell_{1}',\ell_{3},\ell_{3}')$ of linear subspaces of
a $\C^{4}\,{\simeq}\,\Hb^{2}$ with $\ell_{1},\ell'_{1}$ complex lines such that $\ell_1{+}\ell_{1}'$
is a quaternionic line and $\ell_{3},\,\ell_{3}'$  complex hypersurfaces with
$\ell_{3}{\cap}\ell'_{3}\,{=}\,\ell_1{+}\ell'_{1}.$
\end{exam}
\begin{exam} Consider a root system 
\begin{equation*}
 \Rad=\{{\pm}e_{i}{\pm}e_j\mid 1{\leq}i{<}j{\leq}4\}
\end{equation*}
of type $\mathrm{D}_{4}$ and the maximal parabolic $\qt_{\Phi}$
described by the cross-marked Dynkin diagram
   
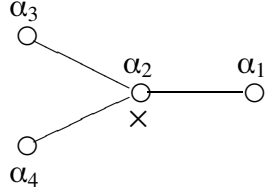
\begin{equation*}
 \xymatrix@R=-.3pc{
 \alphaup_3\\
\!\!
\medcirc\!\!\! \ar@{-}[rdd] \\
&\alphaup_{2} &\alphaup_{1}\\
 & \!\!\medcirc\!\!\! \ar@{-}[r]  & \!\!\medcirc\!\!\! 
\\
&\times\\
\!\!
\medcirc\!\!\! \ar@{-}[ruu]
\\
 \alphaup_{4}}
\end{equation*}
Here $\alphaup_{i}{=}e_{i}{-}e_{i+1},$ for $1{\leq}i{\leq}3$ and $\alphaup_4{=}e_{3}{+}e_{4},$
$\Phi\,{=}\,\{\alphaup_{2}\},$  
\begin{equation*}
 \xiup_{\Phi}(e_{i})= 
\begin{cases}
 1, & i=1,2,\\
 0, & i=3,4.
\end{cases}
\end{equation*}
With the conjugation 
\begin{equation*}
 \sigmaup(e_{1})=e_{4},\;  \sigmaup(e_{2})={-}e_{3},\; \sigmaup(e_{3})={-}e_{2},\; \sigmaup(e_{4})=e_{1},
\end{equation*}
we obtain 
\begin{align*}
 \Qq^{\,c}_{\Phi}\cap\bar{\Qq}^{\,c}_{\Phi}&=\{e_{1}{+}e_{4},\, e_{2}{+}e_{4},\,e_{1}{-}e_{3},\,e_{2}{-}e_{3}\}
 \\
\Qq_{\Phi}\cap\bar{\Qq}_{\Phi}&=\{e_{1}{+}e_{2},\, e_{1}{+}e_{3},\,e_{2}{+}e_{3},\,e_{1}{-}e_{4},
\, e_{2}{-}e_{4}\}
  \\
 \Qq^{\,c}_{\Phi}\cap\bar{\Qq}^{\,c}_{\Phi}&=\{e_{4}{-}e_{3},\, e_{4}{-}e_{2},\,{-}e_{2}{-}e_{3},\,e_{4}{-}e_{1},
 \, {-}e_{3}{-}e_{1}\},\\
 (\Qq^{\,n}_{\Phi}\cap\bar{\Qq}^{\,c}_{\Phi}){\backslash}\Rad_{\,\bullet}&=\{e_{4}{-}e_{2},\, {-}e_{3}{-}e_1\} .
\end{align*}
It is easy to check, using the fact that $\qt_{\Phi}$ is maximal, that $(\st_{\sigmaup},\qt_{\Phi})$ 
is fundamental and weakly nondegenerate; moreover the last line of the equalities above
shows that $(\st_{\sigmaup},\qt_{\Phi})$ is not of the minimal type. To check that $(\st_{\sigmaup},\qt_{\Phi})$
is Levi nondegenerate (i.e. has Levi order $1$) we observe that 
\begin{align*}
 &(e_{4}{-}e_{3}) + (e_{1}{-}e_{4})= (e_{1}{-}e_{3}),\\
 &(e_{4}{-}e_{2}) + (e_{1}{+}e_{2})= (e_{1}{+}e_{4}),\\
&({-}e_{2}{-}e_{3}) + (e_{1}{+}e_{2})= (e_{1}{-}e_{3}),\\
 &(e_{4}{-}e_{1}) + (e_{1}{+}e_{2})= (e_{2}{+}e_{4}),\\
 &({-}e_{1}{-}e_{3}) + (e_{1}{+}e_{2})= (e_{2}{-}e_{3}).
\end{align*}
\end{exam}
\begin{exam}  Consider a root system 
\begin{equation*}
\Rad\,{=}\,\{{\pm}(e_i{+}e_j)\,{\mid}\,1{\leq}i{\leq}j{\leq}3\}\,
{\cup}\,\{{\pm}(e_i{-}e_j)\,{\mid}\,1{\leq}i{<}j{\leq}3\}\end{equation*}
of type $\mathrm{C}_{3}$ and
the cross-marked
Dynkin diagram
 
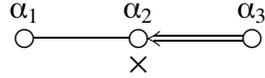
\begin{equation*} 
    \xymatrix@R=-.3pc{
    \alphaup_{1} & \alphaup_{2} & \alphaup_{3}\\
 \!\!\medcirc\!\!\ar@{-}[r]
&\!\!\medcirc\!\! \ar@{<=}[r]
&\!\!\medcirc\!\! 
\\
&\times}
\end{equation*}
with $\alphaup_1{=}e_{1}{-}e_2,$ $\alphaup_{2}{=}e_2{-}e_3,$ $\alphaup_{3}{=}2e_{3}$ and 
$\Phi=\{\alphaup_{2}\}$ so that $$ \xiup_{\Phi}(e_{i})= 
\begin{cases}
 1, & i{=}1,2,\\
 0, & i=3.
\end{cases} $$Consider the conjugation 
\begin{equation*}
 \sigmaup(e_{1})=e_{3},\;\;\sigmaup(e_{2})={-}e_{2},\;\;\sigmaup(e_{3})\,{=}\,e_{1}.
\end{equation*}
We obtain 
\begin{align*}
 \Qq^{\,c}_{\Phi}\cap\bar{\Qq}^{\,c}_{\Phi}&=\{e_{1}{+}e_{3}\},\\
 \Qq^{\,c}_{\Phi}\cap\bar{\Qq}^{\,c}_{\Phi}&=\{2e_{1},\,2e_{2},\, e_{1}{+}e_{2},\, 
 e_{2}{+}e_{3},\, e_{1}{-}e_{3},\, e_{2}{-}e_{3}\}\\
  \Qq^{\,c}_{\Phi}\cap\bar{\Qq}^{\,c}_{\Phi}&=\{ 2e_{3},\,{-}2e_{2},\, e_{3}{-}e_{2},\, 
 e_{1}{-}e_{2},\, e_{3}{-}e_{1},\, {-}e_{1}{-}e_{2}\}.
 \end{align*}
Since the parabolic $\qt_{\Phi}$ is maximal, it is easy to check that $(\st_{\sigmaup},\qt_{\Phi})$
is fundamental and weakly nondegenerate. It has Levi order two, as one can check from 
\begin{align*}
 & 2e_{3} + (e_{1}{-}e_{3})=(e_{1}{+}e_{3}), && {-}2e_{2}+(e_{1}{+}e_{2})+(e_{2}{+}e_{3})=(e_{1}{+}e_{3})\\
 &(e_{3}{-}e_{2})+(e_{1}{+}e_{2})=(e_{1}{+}e_{3}), 
 && (e_{1}{-}e_{2}) + (e_{2}{+}e_{3})=(e_{1}{+}e_{3}), \\
 &(e_{3}{-}e_{1})+2e_{1}=(e_{1}{+}e_{3}), &&
 ({-}e_1{-}e_{2})+2e_{1}+(e_{2}{+}e_{3})=(e_{1}{+}e_{3}).
\end{align*}
\end{exam}
\section{\textsc{Weakly nondegenerate $CR$ manifolds with larger Levi orders}}\label{s3}
Fix any integer $\qq{\geq}1.$ 
In this last section we discuss in detail the  example of a homogeneous
$CR$ manifold $\sfM$ of $CR$ dimension $\qq{+}1$ and $CR$ codimension
$1$ which is fundamental and weakly nondegenerate of Levi order $\qq.$
\par \smallskip
The compact group $\SU(2)$ acts transitively on the complex projective
line $\CP^{1}.$ The homogeneous complex structure of $\CP^{1}$ 
can be defined by
the totally complex
$CR$ algebra $(\su(2),\bt),$ where $\su(2)$ is the real Lie algebra
of anti-Hermitian $2{\times}2$ matrices and $\bt$ a Borel subalgebra
of its complexification $\slt_{2}(\C).$ This $CR$ algebra corresponds to the 
simple
cross-marked Satake diagram 
\begin{equation*}
 \begin{aligned}
\alphaup\\[-7pt]
 \medbullet \\[-7pt]
 \times
\end{aligned}
\end{equation*}
The root system of the complexification $\slt_{2}(\C)$ is $\Rad\,{=}\,\{{\pm}(e_{1}{-}e_{2})\}.$
and we take $\alphaup\,{=}\,(e_{1}{-}e_{2}),$ with fundamental weight $\omegaup\,{=}\,\alphaup/2.$\par
With our usual notation, 
$\Phi{=}\{\alphaup\},$ so that $\bt\,{=}\,\qt_{\Phi}$; moreover $\xiup_{\Phi}(e_{i}){=}({-}1)^{{i+1}}/2$
and $\st_{\sigmaup}{=}\su_{2},$ with 
conjugation
$\sigmaup(e_{1}){=}e_{2},$ $\sigmaup(e_{2})=e_{1}.$ 
\par 
The irreducible finite dimensional complex linear 
representations of $\slt_{2}(\C)$  are indexed by the nonnegative integral  multiples $k{\cdot}\omegaup$ 
of $\omegaup$
and the corresponding irreducible $\slt_{2}(\C)$-module 
$\sfV_{k{\cdot}\omegaup}$ 
can be identified with the space of 
complex homogeneous polynomials of degree $k$ in two indeterminates 
\begin{equation*}
 \sfV_{k\omegaup}=\left.\left\{{\sum}_{h=0}^{k}a_{h}z^{h}w^{k-h}\,\right| \, a_{h}\in\C\right\}.
\end{equation*}
We have 
\begin{equation*}
 \sfV_{k\omegaup}={\bigoplus}_{h=0}^{k}\sfV_{k\omegaup}^{(k-2h)\omegaup},
\end{equation*}
where, for a diagonal $H$ in the canonical Cartan subalgebra of 
$\slt_{2}(\C),$  
\begin{equation*}
 \sfV_{k\omegaup}^{(k{-}2h)\omegaup}=\{\vq\in\sfV_{k\omegaup}\,{\mid}\, H{\cdot}\vq\,{=}\,
 (k{-}2h)\omegaup(H)
 \vq\}=\{a\,{\cdot}\,{z}^{h}w^{k-h}\mid a\in\C\},\;\; 0{\leq}h{\leq}k,
\end{equation*}
are the one-dimensional weight spaces contained in $\sfV_{k\omegaup}.$ 
\par 
Since $\bar{\omegaup}{=}-\omegaup,$ we have $\overline{\sfV}_{k\omegaup}\,{=}\,\sfV_{k\omegaup}.$ 
The anti-$\C$-linear automorphism $\thetaup_{k\omegaup}$
of $\sfV_{k\omegaup}$ defined by the conjugation $\sigmaup$ 
comes from $(z,w)\,{\mapsto}\,({-}\bar{w},\bar{z})$ and therefore 
\begin{equation*}
 \thetaup\left({\sum}_{h=0}^{k}a_{h}z^{h}w^{k-h}\right)={\sum}_{h=0}^{k}({-}1)^{h}
 \bar{a}_{h}
 {w}^{h}{z}^{k-h}
\end{equation*}
Then $\thetaup_{k\omegaup}^{2}$ equals $\id_{\sfV_{k\omegaup}}$ for $k$ even and
${-}\id_{\sfV_{k\omegaup}}$ for $k$ odd. Accordingly, for $k$ even 
$\sfV_{k\omegaup}$ is the complexification of an irreducible 
$(k{+}1)$-dimensional representation of the real type, 
that we will denote by $\sfV^{\R}_{k\omegaup}$; for
$k$ odd is isomorphic to a $2(k{+}1)$-dimensional irreducible representation  
of the quaternionic type of $\su_{2}$
(see e.g. \cite[Ch.IX, App.II, Prop.2]{Bou82}).
\begin{rmk}
Studying irreducible representation of $\su_2$
turns out to be of some interest in quantum physics, 
as they 
arise when considering rotations on fermionic and bosonic systems 
(for more details see \cite[Ch.5 , \S{5}]{Tin64}).\par
\end{rmk}
The subspace 
\begin{equation*}
 \sfV^{-}_{k\omegaup}={\bigoplus}_{k{<}2h{\leq}{2k}}\sfV_{k\omegaup}^{(k-2h)\omegaup}
\end{equation*}
is a $\bt$-submodule of $\sfV_{k\omegaup}$ and we can consider the
semidirect sum $\bt\,{\oplus}\,\sfV_{k\omegaup}^{-}$ as a subalgebra
of the abelian extension 
$\slt_{2}(\C)\,{\oplus}\,\sfV_{k\omegaup}$  
(cf.  e.g. \cite[Ch.VII,\S3]{HiSt}).
We may consider the map $\SL_{2}(\C)\to\CP^{1}$ associated to our choice
of a Borel subalgebra $\Bf$ as a principal bundle with structure group $\Bf.$
Then the Lie pair 
$(\slt_{2}(\C)\,{\oplus}\,\sfV_{k\omegaup},\bt\,{\oplus}\,\sfV_{k\omegaup}^{-})$
defines a complex holomorphic vector bundle $\sfE_{k}$ with 
base $\CP^{1}$ and typical fiber 
$\sfV_{k\omegaup}/\sfV_{k\omegaup}^{-}{\simeq}{\bigoplus}_{2h{\leq}k}\sfV_{k\omegaup}^{(k-2h)\omegaup}$ 
(this is an example of  Mostow fibration, see \cite{MN00} for more details).

\begin{prop} Let $\qq$ be any positive integer. Then 
\begin{equation}\label{exq}
(\gt_{\R},\qt_{\Phi}')=
(\su_{2}{\oplus}\,\sfV^{\,\R}_{2\qq\omegaup},
\bt\,{\oplus}\,\sfV_{2\qq\omegaup}^{\,-})\end{equation}
 is the $CR$ algebra of a
 $CR$ manifold 
 $\sfE_{2q}$, 
 of $CR$ dimension $q{+}1$ and
 $CR$ codimension $1$, which is fundamental and weakly nondegnenerate of 
 Levi order~$\qq.$ 
\end{prop} 
\begin{proof} We have 
\begin{equation*}
\bar{\sfV}_{2\qq\omegaup}^{\,-}{=}\sfV_{2\qq\omegaup}^{\,+}= 
{\bigoplus}_{h=1}^{\qq}\sfV_{2\qq\omegaup}^{2h\omegaup} 
\quad\text{and}\quad
 \sfV_{2\qq\omegaup}= \sfV_{2\qq\omegaup}^{\,-} \oplus\sfV_{2\qq\omegaup}^{\,0} 
 \oplus\sfV_{2\qq\omegaup}^{\,+}.
\end{equation*}
\par
 If $Z_{\alphaup},Z_{-\alphaup},H$ is the canonical basis of $\slt_{2}(\C)$
 and $\wq$ a nonzero vector of $\sfV_{2\qq\omegaup}^{-2\qq\omegaup},$
then the images of $X_{-\alphaup},\wq,X_{\alphaup}\wq,\hdots,X^{\qq-1}_{\alphaup}\wq$
generate $\qt'_{\Phi}/(\qt'_{\Phi}{\cap}\bar{\qt}'_{\Phi}).$ Since 
\begin{equation*}
 [\underset{\text{$h$ times}}{\underbrace{X_{\alphaup},\hdots,X_{\alphaup}}},X_{\alphaup}^{\qq-h}\wq]
 =X_{\alphaup}^{\qq}\wq \in \sfV^{\,0}_{2\qq\omegaup}\setminus\{0\},
\;\; [X_{\alphaup}^{\qq+1}\wq,X_{-\alphaup}]=-2 X_{\alphaup}^{\qq}\wq 
\in \sfV^{\,0}_{2\qq\omegaup}\setminus\{0\}
\end{equation*}
we obtain that $\sfE_{2\qq}$ is fundamental and weakly nondegenerate. 
With the notation of the previous section, we have 
 $\Qq_{\Phi}^{\,c}{\cap}\bar{\Qq}_{\Phi}{=}\{\alphaup\},$
 with $\xiup_{\Phi}(\alphaup){=}1$ and $\xiup_{\Phi}(-2j\omegaup){=}{-}j.$
 Since $\gt/(\qt'_{\Phi}{+}\bar{\qt}'_{\Phi})$ is generated by the image of
 $\sfV_{2\qq\omegaup}^{\,0},$  by the above considerations 
 the Levi order of an element of $\sfV^{-2j\omegaup}_{2\qq\omegaup}$ equals $j.$ 
This shows that the Levi order of $(\gt_{\R},\qt'_{\Phi})$ is $\qq.$ 
\end{proof}

In an analogous way we can also prove 
\begin{prop}
 For each positive integer $\qq$,
 the homogeneous $CR$ manifold $\sfE_{2\qq}$ is    
 contact nondegenerate of order $\qq.$ \qed
\end{prop} 
\begin{rmk}
Representations
$\sfV_{k\omegaup}$ with an odd $k$ are canonically associated with 
complex holomorphic vector bundles $\sfE_{k},$ 
of rank $(k{+}1),$ with base
$\CP^{1}.$ 
\end{rmk}
\bibliographystyle{amsplain}
\renewcommand{\MR}[1]{}

\bibliography{homog}

\end{document}